%% file: ms.tex
\documentclass{amsart}
\pdfoutput=1
\usepackage{amsmath, amsthm, amssymb,  color}
\usepackage{hyperref}
\usepackage{graphicx}
\usepackage[utf8]{inputenc}
\usepackage[all]{xypic}
\usepackage{csquotes}
\usepackage{adforn}
\usepackage{mathtools}
\usepackage{tikz}
\usetikzlibrary{intersections,matrix,arrows,calc,decorations,decorations.pathmorphing,decorations.markings,fadings,positioning,patterns,shapes}
\usepackage{pgfplots}
\usepackage{caption}
\usepackage{subcaption}
\usepackage[style=alphabetic,sorting=nty,backend=biber,backref,language=english]{biblatex}
\usepackage{xcolor}
\hypersetup{
    colorlinks,
    linkcolor={red!50!black},
    citecolor={blue!50!black},
    urlcolor={blue!80!black}
}
\input{preamble}

\usepackage{color}
\newif\ifincludeprevious

\includeprevioustrue

\newtheorem{teorema}{Theorem}
\numberwithin{teorema}{section}
\newtheorem{proposicion}[teorema]{Proposition}
\newtheorem{lema}[teorema]{Lemma}
\newtheorem{corolario}[teorema]{Corollary}
\newtheorem{definicion}[teorema]{Definition}
\newtheorem{observacion}[teorema]{Remark}

\theoremstyle{definicion}
\newtheorem*{teorema*}{Theorem}

\date{}
 
\title{Horizontal strips and spaces of quadratic differentials}
\author{Rom\'an Contreras}

\begin{document}
\maketitle

\tableofcontents
\input{intro}

\input{quaddiff}
\input{spectral}

\input{strips}

\input{spaces}

\printbibliography
\end{document}

%% file: preamble.tex
\tikzset{->-/.style={decoration={
  markings,
  mark=at position #1 with {\arrow{>}}},postaction={decorate}}}

\definecolor{color1}{RGB}{33,78,172}
\definecolor{color2}{RGB}{242,133,16}
\definecolor{color11}{RGB}{245,245,245}
\definecolor{color12}{RGB}{232,232,232}
\definecolor{color13}{RGB}{122,11,94}

\colorlet{fondoResaltado}{color11}
\colorlet{orange}{color1}
\colorlet{purple}{color11}

\DeclareMathOperator{\GeneralLinear}{GL}
\DeclareMathOperator{\Tangente}{T}
\DeclareMathOperator{\Id}{Id}

\DeclareMathOperator{\Cuad}{\mathbf{Quad}}
\DeclareMathOperator{\dimension}{dim}
\DeclareMathOperator{\rango}{rank}
\DeclareMathOperator{\grad}{deg}
\DeclareMathOperator{\ord}{ord}

\DeclareMathOperator{\im}{Im}
\DeclareMathOperator{\re}{Re}

\DeclareMathOperator{\CritInf}{Crit_\infty}
\DeclareMathOperator{\Divisor}{Div}
\DeclareMathOperator{\PDivisor}{PDiv}
\DeclareMathOperator{\Hom}{Hom}
\DeclareMathOperator{\pic}{Pic}
\DeclareMathOperator{\Aut}{Aut}
\renewcommand{\H}{\mathrm{H}}

\newcommand\hatH{\widehat{\textrm{H}}}

\newcommand{\ie}{\textit{i.\,e.~}}
\newcommand{\apriori}{\textit{a priori~}}

\newcommand{\bigslant}[2]{{\raisebox{.8em}{$#1$}\left/\raisebox{-.8em}{$#2$}\right.}}

%% file: intro.tex
\section*{Introduction}

This work is an attempt to give a complete albeit short account of
the theory needed to define certain families of quadratic differentials
(and a natural parametrization of such families).

The motivation for this work arose after several attempts
at understanding the paper \cite{bridgeland2015}.
Given the extension and depth of the paper, we choose to
work just on a portion of it, and focus in the particular case of
GMN differentials with double poles and no saddle trajectories.

We now give a brief panoramic view of \cite{bridgeland2015},
as well as the scope and reach of this text and its relation
to the paper.

A \emph{quadratic differential} on a Riemann surface is a
meromorphic section of the tensor square of the canonical bundle
of the surface.
Locally it can be written as 
\( f(z)dz^2\) where  \( z\) is a holomorphic coordinate and 
\( dz^2\) is a shorthand for \( dz\otimes dz\).

Quadratic differentials, which \apriori are part of the complex geometry realm,
have several related \emph{real} geometric structures.
Every differential gives rise to a flat Riemannian metric with conical singularities
on the Riemann surface. Also, they define a circle of
foliations by geodesics of this metric.
Both of these structures are central to the study of
quadratic differentials, they are treated in section \ref{cap:geometria-diferenciales}.

One of the fundamental results about this geometric structures
is that they are locally trivial except at a finite number of points:
the metric and foliation are isomorphic to the Euclidean metric and the
foliation by horizontal lines in  \( \mathbb{C}\).
This is equivalent to the fact that it is always possible to
find charts where the differential is equal to \( dz^2\).

From the whole circle of foliations of a quadratic differential,
there is a distinguished one: the horizontal foliation.
This foliation corresponds to the foliation by horizontal lines of the complex plane.
The leaves of this foliation are called the horizontal trajectories of the
quadratic differential.
The global structure of the horizontal trajectories
can be very complicated and is one of the main themes of the theory.

Returning to the complex geometry context,
the quadratic differentials are the ``quadratic'' analogue
of meromorphic differential \( 1-\)forms.
They can also be integrated along curves, but in this case,
it is necessary to first take a local square root of the differential.
A slightly technical point is that the square root of
a quadratic differential is not globally defined.
Rather, it is more natural to define the square root on a two sheeted
branched covering of the surface.
This covering is called the \emph{spectral cover}.
Once we construct the spectral cover, it its possible
to define a group morphism from (a subgroup of) the first
homology group of the spectral cover to the complex numbers, given
by integration of the square root of the differential.
This morphism is the analogue of the periods
of a meromorphic \( 1-\)form.
We will discuss the construction of the spectral cover
and its main properties in section \ref{cap:espectral-hom}.

If the quadratic differential satisfies certain properties
(does not have trajectories connecting zeroes, and only has double poles)
it is possible to split the surface in a union of \emph{horizontal strips}
by cutting along some horizontal trajectories.
In this horizontal strips, the differential is trivial.
Every horizontal strip gives rise to a distinguished homology class
on the spectral cover called a \emph{standard saddle class}.
The collection of all the standard saddle classes form a basis for
the \emph{hat homology} group. This is relevant because the
cohomology class of the square root of the differential is fully
determined by its value on the hat homology group.
The details of the horizontal strip decomposition and the standard saddle classes
can be found in section \ref{sec:franjas}.

Finally, in section \ref{sec:espacios} we define the spaces of \emph{framed differentials} of a given polar type, and prove that the \emph{period mapping},
gives an isomorphism to \(\mathbb{H}^n\) for some \(n\) that only depends on
the genus of the surface and the orders of the poles of the differentials.

Quadratic differentials have been studied as early as the early twentieth century\footnote{See \cite[p.~VI]{strebel-quadratic}} and although their study originated
in the context of functional analysis and variational calculus,
nowadays they are relevant in several areas.

The first big success in the theory was in the work of
Oswald Teichm\"uller in connection to the theory of quasiconformal
transformations.
This eventually led Alfors and many others to the construction of
Teichm\"uller spaces and an analytical realization of the moduli space
of Riemann surfaces (which is the right way to topologize the set of
isomorphism classes of Riemann surfaces).

Later, towards the second half of the twentieth century
began the study of the moduli spaces of quadratic differentials
(which are related, but not the same as the moduli space of Riemann surfaces)
\footnote{And also the spaces of meromorphic \( 1-\)forms or \emph{Abelian differentials}. This spaces are closely related to moduli spaces of quadratic differentials.}
\ie  spaces whose points are in bijection with the isomorphism classes
of quadratic differentials.
Usually there is some restriction on the orders of the poles
o the differentials.

This spaces have a very rich and complicated geometry.
They are stratified singular spaces (orbifolds).
Besides, they naturally have a dynamical structure: there is an
action of the group \(\GeneralLinear_+(2,\mathbb{R}) \),
and also a flow called the \emph{Teichm\"uller flow},
which preserves a natural measure.%
\footnote{See, for example, the papers of Kontsevich-Zorich and Lanneau about
the number of connected components of the strata \cite{kontsevich2003,lanneau2008}; the paper
of Eskin-Okounkov about the volume of the moduli space \cite{eskin2001};
and the paper of Avila-Viana about the Kontsevich-Zorich conjecture \cite{avila2005}, among many other works about this spaces.}

One of the most recently studied topics of this theory
concerns its relation to \emph{stability conditions} on
\emph{triangulated categories}.%
\footnote{For a brief introduction to derived and triangulated categories, see  \cite{thomas2000derived}.}

Given a stability condition \( \sigma\) on a category \( \mathcal{C}\),
there is a group morphism
\(Z_\sigma: K(\mathcal{C})\rightarrow \mathbb{C}\) called the central charge of the stability condition,
where \( K(\mathcal{C})\) is the \emph{Grothendieck group} of the category.
Let \( Stab(\mathcal{C})\) be the set of all stability conditions on a category \( \mathcal{C}\).
In \cite{bridgeland2007} it is proven that \( Stab(\mathcal{C})\) has a natural topology
with which it becomes a complex manifold,
and that the function that maps every stability condition
to its central charge is a local isomorphism.

In the year 2007, Maxim Kontsevich made the following remarks:
\blockquote[\cite{kontsevich2007donaldson}]{
Hence, $Stab(\mathcal{C})$ is a complex manifold, not necessarily connected.
Under our assumptions one can show that the group $Aut(\mathcal{C})$ acts properly discontinously on $Stab(\mathcal{C})$.
On the quotient orbifold $Stab(\mathcal{C})/Aut(\mathcal{C})$ there is a natural non-holomophic action of $\GeneralLinear_+(2,\mathbb{R})$ arising from linear transformations of $\mathbb{R}^2\cong \mathbb{C}$ preserving the standard orientation. A similar geometric structure appears on the moduli spaces of holomorphic Abelian differentials,\textelp{}}

The next year, in a paper by  Kontsevich and Yan Soibelman, we can find
the next observation:

\blockquote[\cite{kontsevich2008stability}]{
Geometry similar to the one discussed in this paper also appears in the theory of moduli spaces of holomorphic abelian differentials \textelp{} 
The moduli space of abelian differentials is a complex manifold, divided by real “walls” of codimension one into pieces glued from convex cones. It also carries a natural non-holomorphic action of the group $\GeneralLinear_+(2,\mathbb{R})$. There is an analog of the central charge $Z$ in the story. It is given by the integral of an abelian differential over a path between marked points in a complex curve. This makes plausible the idea that the moduli space of abelian differentials associated with a complex curve with marked points, is isomorphic to the moduli space of stability structures on the (properly defined) Fukaya category of this curve.}

From a completely different perspective, quadratic differentials appeared
in the study of asymptotic solutions of partial differential equations,
more precisely the WKB method.%
\footnote{The WKB method (named after Wentzel–Kramers–Brillouin) originally
arose in quantum mechanics as a method
to obtain approximate solutions to the Schr\"odinger equation.}
In this context, the horizontal trajectories arising from a zero of the differential are called \emph{Stokes lines} and its study gave rise to the so called
\emph{Stokes phenomenon}.
In \cite{kawai}  \emph{Stokes graph}
gives rise to a triangulation of the Riemann sphere.

Some years later, in \cite{neitzke-wall-crossing},
Gaiotto, Moore, and Neitzke prove that under certain hypothesis
every (meromorphic) quadratic differential defines
an ideal triangulation of the underlying surface.
In this triangulation the arcs are given by choosing some horizontal trajectories
and the vertices are exactly the double poles of the differential.
This triangulation is called the \emph{WKB triangulation}.
Also, they conjecture about some possible connection between
spaces of quadratic differentials and the \emph{wall-crossing formula}
of Kontsevich and Soibelman.

For some years now, it has been known that
every triangulation of a surface gives rise to a
triangulated category.
For a panoramic view of the necessary steps to
define said category, see  \cite{2013labardini}.

Finally, synthesizing the previous results, 
Bridgeland and Smith prove in \cite{bridgeland2015} that the space of
quadratic differentials on a Riemann surface is
isomorphic to the space of stability conditions on the
category determined by any WKB triangulation of any quadratic differential on
the Riemann surface.
One of the delicate points of the paper is
proving the invariance of the triangulated category
under perturbations of the quadratic differential.

In this text we present the necessary tools to define the 
\emph{WKB triangulation}
for a GMN differential without saddle trajectories.
Also, we develop the necessary material to prove that
the space of GMN differentials without poles of order greater than two
and without saddle trajectories naturally decomposes as
a union of sets parametrized by  \( \mathbb{H}^n\).
This last assertion is approximately equivalent to
the contents of Section \( 4.5\) of \cite{bridgeland2015}.

Topologically, this families of differentials correspond to
cells with complement of codimension one . Thus they give rise to
a wall and chamber decomposition of the space of quadratic differentials.

It is not hard to prove that in every cell, the WKB triangulation is constant,
hence, to prove the invariance of the triangulated category
under changes of the quadratic differentials,
it is enough to give a detailed analysis of the change in the
triangulation when going from one cell to another,
and how this change relates to the triangulated categories involved.
This is precisely one of the key arguments in \cite{bridgeland2015}.

%% file: quaddiff.tex
\section{Geometry of quadratic differentials}%
\label{cap:geometria-diferenciales}

In this section we define the main objects of study: quadratic differentials.
Also, we reproduce some of the classical results about the geometry of quadratic differentials.

All the theorems in this section are well-known and can be found in several sources.
We choose to follow \cite{strebel-quadratic}.

To avoid unnecessary references to the same text, 
whenever a theorem doesn't have an explicit reference, it should be understood
that the theorem can be found in the aforementioned source.

All the other theorems have an explicit reference to the literature.

\subsection{Quadratic differentials}

Let \( X\) be a compact Riemann surface and  \( K_X\) be the canonical bundle over \( X\),
\ie , the line bundle over \( X\) such that the fiber \( K_x\) over any point 
\( x \in X\) is the cotangent space \(\Tangente_x^* X\) to \( X \) at \(x\). \footnote{More precisely, the \emph{complex} cotangent space.}
	
Let's recall that the holomorphic sections (resp. meromorphic) of \(K_X\) 
are precisely the holomorphic \( 1-\)forms (resp. meromorphic) over 
\( X\).\footnote{ The holomorphic differentials are also known as \emph{Abelian differentials}.}

Let \(Q_X:=K_X^{\otimes 2}\) be the tensor square of \( K_X\).
\(Q_X\) is a line bundle over \( X\) such that for any point \( x\in X\), the fiber \(Q_x\)
is canonically isomorphic to
\( \Tangente^*_x X\otimes_\mathbb{C}\Tangente^*_x X\).

\begin{definicion}
	A \emph{holomorphic quadratic differential} (resp. meromorphic) over  \(X\) is a holomorphic section (resp. meromorphic) of \( Q_X\).
\end{definicion}

Whenever we speak of line bundles (or more generally, fiber bundles),
there are two equivalent ways of doing so:

The first is describing the bundle in global terms,
that is, through a \(2\)-dimensional complex manifold \(L\)
and a projection \( p:L \rightarrow X\) such that \(p\) is a submersion and
\( L_x:= p^{-1}(x)\) has the structure of a one dimensional complex vector space.

The second is describing the bundle in local terms, \ie, by means of an open covering \(  \mathcal{U}\) of \( X\)
and \emph{transition functions}  \(\{ \phi_{U,V}\}_{(U,V)\in  \mathcal{U}\times  \mathcal{U}}\)
such that \[ \phi_{U,V}:U \cap V\rightarrow\GeneralLinear(1,\mathbb{C})\cong \mathbb{C}^*\]
and the transition functions satisfy certain \emph{cocycle} conditions.

The second form allows us to give a correspondence between isomorphism classes of line bundles over \(X\)
and the elements of the first sheaf cohomology group \( \H^1 (X,\mathcal{O}^*_X)\) of the sheaf of non-vanishing holomorphic functions \( \mathcal{O}^*_X\).

Let's write \(\pic(X)\) for the set of isomorphism classes of line bundles over \(X\).
The tensor product operation between line bundles gives rise to a group structure on \( \pic(X)\),
such that the identity element is the trivial bundle \( \mathbb{C}_X:= X\times \mathbb{C}\),
and the inverse of a bundle \( L\) is the dual bundle \( L^*\).

This group is called the \emph{Picard group} of \( X\).

The above correspondence sets up a group isomorphism 
\[ \pic(X)\cong \H^1(X,\mathcal{O}^*_X).\]

The same way we can describe line bundles in local or global terms,
sections of line bundles can also be described in both ways.

From the global viewpoint, a section \( \sigma\) of a bundle \( p:L \rightarrow X\)
is a holomorphic function
\( \sigma: X\rightarrow L\) such that \( p\circ \sigma =\Id_X\).
In the local picture, a section \( \sigma \) of a bundle given by transition functions
\(\{ \phi_{U,V}\}_{(U,V)\in  \mathcal{U}\times  \mathcal{U}}\) defined on an open covering \( \mathcal{U}\)
is a collection of holomorphic functions \(\{\sigma_U:U\rightarrow \mathbb{C}\}_{U\in  \mathcal{U}}\)
that satisfy
\[ \sigma_U=\sigma_V\phi_{U,V}\]
in the common domain
\( U\cap V\).

If \( \sigma\) is a section of a line bundle \( L\) and  \( \tau\) is a section of  \( L^\prime\)
then, \( \sigma\otimes \tau\) is a section of \( L\otimes L^\prime\).

Moreover, if \(U\) is an open set where \(L\) and \(L^\prime\) are trivial, trivializations of \(L\) and \(L^\prime\)
induce a trivialization of \( L\otimes L^\prime\), and
\( (\sigma\otimes \tau)_ U= \sigma_ U\tau_ U\), \ie
tensor product of sections corresponds to the usual product of functions.

Let \( z: U \rightarrow \mathbb{C}\) be a local chart of \( X\). For every point \( x\in  U\)
the differential of \( z\) is a  \( \mathbb{C}\)-linear isomorphism \( dz_p:\Tangente_p X\rightarrow \Tangente_{z(p)}\mathbb{C}\cong \mathbb{C}\), \ie, \( dz_p \in \Tangente_p^* X\).
With this observation we can conclude that \( dz\) is a non-vanishing section  of \(K\) over  \(  U\).

Since \( K_p\) is one dimensional, any other section is a (pointwise) sacalar multiple of \( dz\),
so that, over \(  U\) \( \sigma=f(z)dz\) where \( f\) is a holomorphic function.

The same way, \( dz\otimes dz\) is a non-vanishing section of \( Q\),
so that any section \( \varphi\) of \( Q\) can be written as
\( f(z)dz\otimes dz\) locally. Usually we will write  \(  dz^2\) instead of \( dz\otimes dz\).

If \( \phi:  U \rightarrow V\) is a biholomorphic mapping between two open sets of the complex plane,
and \(\omega=  f(z)dz\) is a holomorphic
\( 1-\)form
on \( V\), then, the \emph{pullback} of \( \omega\) by \( \phi\)
is the holomorphic \(1-\)form \( \phi^*(\omega)= f(\phi(z))\frac{d\phi}{dz}dz\).

If \( \varphi= f(z)dz^2\) is a quadratic differential, we define
the \emph{pullback} of \( \varphi\) by \( \phi\) as \( \phi^*(\varphi)= f(\phi(z))\left (\frac{d\phi}{dz}\right)^2dz^2\).

This is the only possible definition if we require that the pullback operation to be compatible with the tensor product of sections of bundles. So, if \( \omega\) and \( \omega^\prime\) are two holomorphic \( 1-\)forms, then
\( \phi^*(\omega\otimes \omega^\prime)=\phi^*(\omega)\otimes\phi^*(\omega^\prime)\).

So, if \( z: U \rightarrow \mathbb{C}\) is a chart ,
\( dz^2\) is the quadratic differential defined previously, and \( \varphi\) is another quadratic differential
such that \( \varphi=f(z)dz^2\), then  
\[ (z^{-1})^*(\varphi)=f(z)dz^2\]
where the right hand side is a quadratic differential on \( \mathbb{C}\).

From this, we can conclude that if \( z:  U \rightarrow \mathbb{C}\) and
\( w: V\rightarrow \mathbb{C}\) are two charts, \( \varphi\) is a quadratic differential such that
\( \varphi = f(z)dz^2 = g(w)dw^2\), then 
\[ f(z)dz^2= f(z(w))\left(\frac{z(w)}{dw}\right)^2dw^2\]
so that a rule analogous to the chain rule is also valid for quadratic differentials.

\subsection{Metric and foliation of a quadratic differential}

As we have already mentioned, quadratic differentials give rise to a
foliation and a Riemannian metric, both with singularities, on the Riemann surface.

Let \( X\) be a Riemann surface and  \(\varphi\) a meromorphic quadratic differential%
\footnote{This text is mainly concerned with differentials with double poles.}%
over\( X\).
\begin{definicion}
	A \emph{critical point} of \(\varphi\) is a zero or a pole of \(\varphi\).
	A \emph{regular point} is a point that is not critical.
	A critical point is \emph{finite} if it is a zero or a pole of order one. If it is a pole of order at least two
	it is a \emph{infinite} critical point.
\end{definicion}

Let \( p\) be a regular point of \(\phi\). We can find a contractible neighbourhood \( U\) of \( p\) with no critical points of  \( \varphi\)
and where \( \varphi= f(z)dz^2\).
In such a neighbourhood we define the function:
\[ \psi(q):= \int_p^q \sqrt{f(z)}dz\]
where \( \sqrt{f(z)}\) is a branch of the square root of \( f\).

The function \(\psi\) is well defined and holomorphic%
\footnote{It does not depend on the path from \( p\) to \( q\) chosen to perform the integration.}
because \(\sqrt{f(z)}dz \) is a holomorphic \( 1\)-form, hence, closed.

Furthermore, it has non-vanishing derivative, so that in a possibly smaller neighbourhood of \(p\)
it is a biholomorphism.

A simple calculation shows that
\( (\psi^{-1})^*(\varphi)= dz^2\) the natural quadratic differential of \( \mathbb{C}\).

We have proven:

\begin{proposicion}\label{cartas-normales}
	Around any regular point \( p\) of \( \varphi\) there exists a chart 
{\(\psi :U\rightarrow \mathbb{C}\)}
such that \( (\psi^{-1})^*(\phi)= dz^2\).

Besides, any chart with this property can be written in the form
\[ \psi(q)=\int_p^q\sqrt{f(z)}dz+b\]
where \( b\) is a constant and \( \varphi=f(z)dz^2\).
\end{proposicion}

The charts that satisfy the conclusion of the proposition above
play a major role in the definition of the metric and the foliation.

\begin{definicion}
	We will call \emph{normal chart} (for \(\varphi\) ) any chart \( z\)
	of \( X\) such that \( \varphi=dz^2\).
\end{definicion}

If \( \psi:U\rightarrow V\) is a biholomorphism
between two open sets of the complex plane such that \( \psi^*(dz^2)=dz^2\), then
\(( \frac{d\psi}{dz})^2dz^2=dz^2\) so that \(( \frac{d\psi}{dz})^2=1 \)
hence \( \psi(z)=\pm z +b\).

This last observation implies that the change of coordinate maps between two normal charts
are precisely
\[ z\mapsto \pm z+b.\]
In any case, the Euclidean metric and the horizontal foliation%
\footnote{The horizontal foliation of \( \mathbb{C}\) is the foliation whose leaves are the horizontal lines.
	In fact, \emph{every} foliation by parallel lines is invariant under such transformations.
	In any case, the most important foliation in this text is the horizontal one.}
of \( \mathbb{C}\)
are invariant under such change of coordinate maps.

\begin{definicion}\label{def:metrica-foliacion}
	The metric defined by \( \varphi\) is the metric that equals the Euclidean metric in any normal chart.

	In the same way, the (horizontal) foliation defined by \( \varphi\) is given by the pullback of the
	horizontal foliation by normal charts.
The leaves of the horizontal foliation are called \emph{horizontal trajectories} of the differential.
\end{definicion}

Both structures are uniquely defined and invariant under normal coordinate changes.

Using proposition \ref{cartas-normales} it's not hard to show that the metric is locally given by
 \( |\sqrt{f(z)}||dz|\) and that the foliation agrees with the level sets of the function
 \[ q\mapsto \im\left(\int_p^q\sqrt{f(z)}dz\right)\]
where \( \im(\cdot)\) is the imaginary part.

Since the metric is locally isometric to the Euclidean metric,
its clear that it is a flat metric.
The leaves of the horizontal foliation (with an appropriate parametrization) are geodesics of the metric.

\begin{observacion}
	So far we have only defined the metric and the foliation
	on the complement of the critical points.
	The metric can be extended to the finite critical points
	(although it will no longer be a Riemannian metric in the usual sense)

	In a certain way, the infinite critical points are points at infinity,
	because any path from any regular point to an infinite critical point
	necessarily has infinite length.
		 \end{observacion}

\subsection{Local geometry of the horizontal foliation}

Because of the way in which the foliation and metric are defined, it is clear that both structures are
locally trivial around any regular point.
From this fact, we can conclude that, in order to understand the local behaviour of these structures,
we only need to concentrate on the critical points.

In this subsection we will classify explicitly the local behaviour of the metric and the foliation.

First we state the main result and then we give a brief sketch of the proof.

\begin{teorema}\label{teo:criticos}
	Let \( p\) be a critical point of \( \varphi\). If the order of \( \varphi\) at \( p\) is \( n\)
	with \( n>0\) or \( n=2k+1<0\), \ie, \( p\) is a zero or a pole of odd order,
	then there is a chart \( z\) centered at \(p\) such that 
	\[ \varphi=Cz^ndz^2,\]
	where \( C\) is a positive constant.

	If \( \varphi\) is of order \( n=2k<0\) at \( p\), \ie, \( p\) is an even order pole,
	then there exists a chart \( z\) centered at \( p\) such that 
	\[ \varphi= \left( Cz^k +bz^{-1} \right)^2dz^2,\]
	where \( C\) is a positive constant and  \( b\) is a complex number.
\end{teorema}

The previous theorem allows us to give a complete classification of the foliation and metric near critical points:

\begin{teorema}\label{teo:forma-criticos-finitos}
	Let \( \varphi=z^ndz^2\) with \( n\geq -1\).
	Then, the \( n+2\) rays from the origin with constant angles 
	\( k\frac{2\pi}{n+2}\) where \( k=0,1,\ldots,n+1\) are leves of the foliation of \( \varphi\).

	This rays divide the plane in  \( n+2\) sectors, each of this sectors have a metric and foliation isomorphic to the Euclidean metric and horizontal foliation
of the upper half-plane.
The isomorphism is given by the biholomorphism \( z\mapsto \left(\frac2{n+2}\right)^2z^{\frac{n+2}2}\)
	that is well-defined in each one of the sectors.
\end{teorema}

It is possible to use the theorem \ref{teo:forma-criticos-finitos} to describe explicitly the foliation around
the origin for the quadratic differentials \( \varphi=z^ndz^2\) with \( n\geq -1\).

Figure \ref{fig:foliacion-criticos-finitos} shows some of the horizontal trajectories
for those quadratic differentials with \( n<3\).

\begin{figure}[h]
\centering
\begin{subfigure}[t]{0.5\textwidth}
	\includegraphics{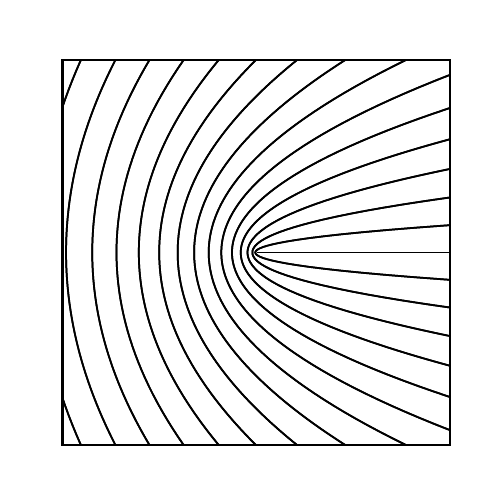}
	\caption{$\frac1zdz^2$}
\end{subfigure}
~
\begin{subfigure}[t]{0.5\textwidth}
	\includegraphics{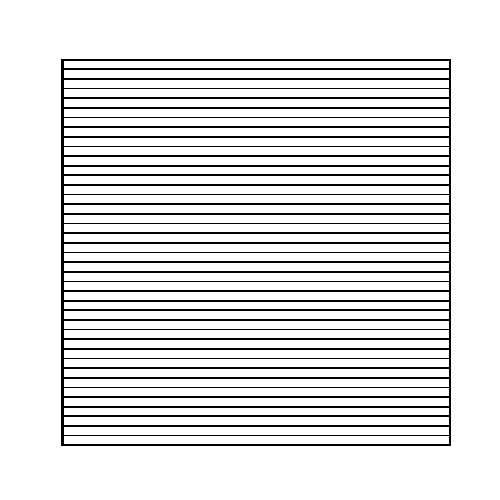}
	\caption{$dz^2$}
\end{subfigure}

\begin{subfigure}[t]{0.5\textwidth}
	\includegraphics{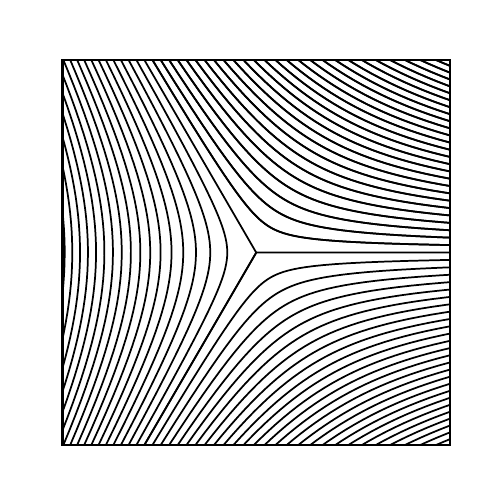}
	\caption{$zdz^2$}
\end{subfigure}
~
\begin{subfigure}[t]{0.5\textwidth}
	\includegraphics{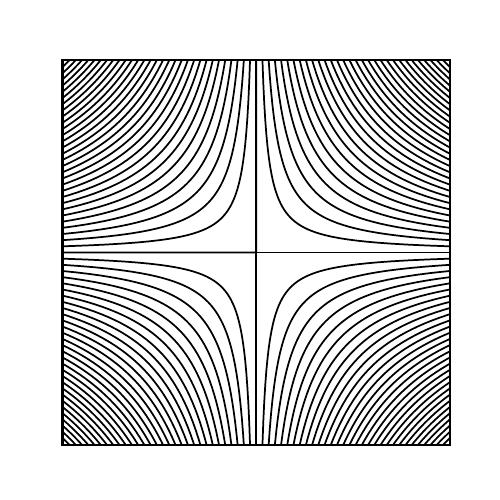}
	\caption{$z^2dz^2$}
\end{subfigure}
\caption{Horizontal trajectories of $z^ndz^2$. Here the horizontal trajectories are at a fixed distance from each other
in the metric given by the differential. This allows us to appreciate the behaviour of the metric near a finite critical point.}\label{fig:foliacion-criticos-finitos}
\end{figure}

In a similar way to the previous result we have the next theorem:

\begin{teorema}\label{teo:forma-criticos-infinitos}
	Let \( \varphi=z^{-n}dz^2\) con \( n\geq 3\).
	The \( n-2\) rays from the origin with constant angles 
	\( k\frac{2\pi}{n-2}\) where \( k=0,1,\ldots,n-3\) are leaves of the foliation of \( \varphi\).

	The rays divide the plane in \( n-2\) regions, each of which has a metric and foliation 
	isomorphic to the Euclidean metric and horizontal foliation of the upper half-plane.
	The isomorphism is given by the biholomorphism
	\( z\mapsto \left(\frac2{2-n}\right)^2z^{\frac{2-n}2}\)
	This map is well-defined in each region.
\end{teorema}

Unlike the theorem \ref{teo:forma-criticos-finitos},
now every horizontal trajectory, with the possible exception of the rays, converge to the origin in both directions.

Figure \ref{fig:foliacion-criticos-infinitos} shows some of the horizontal trajectories for the differentials \( z^{-n}dz^2\) with \( n<7\).
\begin{figure}
\centering
\begin{subfigure}[t]{0.5\textwidth}
	\includegraphics{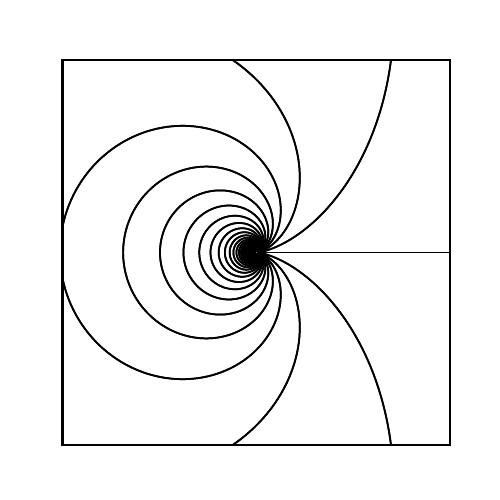}
	\caption{$z^{-3}dz^2$}
\end{subfigure}
~
\begin{subfigure}[t]{0.5\textwidth}
	\includegraphics{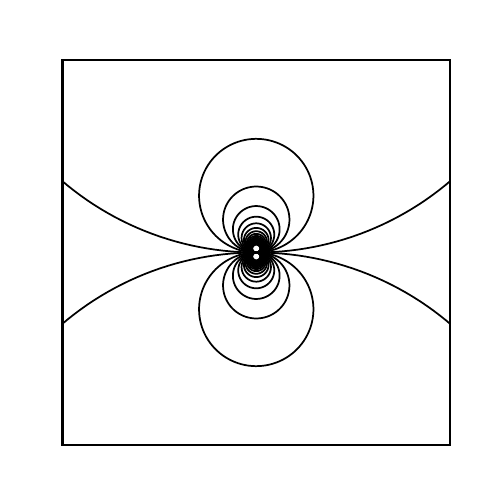}
	\caption{$z^{-4}dz^2$}
\end{subfigure}

\begin{subfigure}[t]{0.5\textwidth}
	\includegraphics{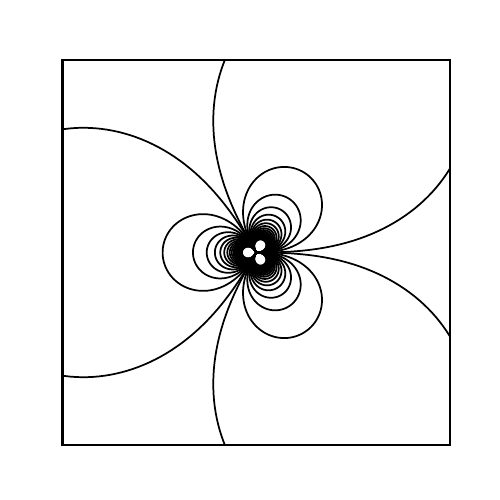}
	\caption{$z^{-5}dz^2$}
\end{subfigure}
~
\begin{subfigure}[t]{0.5\textwidth}
	\includegraphics{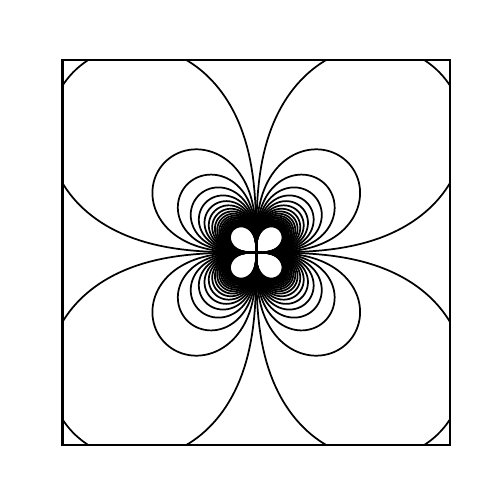}
	\caption{$z^{-6}dz^2$}
\end{subfigure}
\caption{Horizontal trajectories of $z^{-n}dz^2$. Here the trajectories are also at a fixed distance from one another.}\label{fig:foliacion-criticos-infinitos}
\end{figure}

In the case of differentials \( \left(z^{-n} +\frac b{z}\right)^2dz^2\)
with \( b\neq 0\) y \( n>1\)
it is no longer possible to find an explicit biholomorphism 
between a sector bounded by rays and the upper half-plane.
However, it is possible to prove that there is a neighbourhood of
the origin in which the foliation is isomorphic to the horizontal foliation
of \( z^{-2n}dz^2\).

Finally, the next theorem addresses the case
\( bz^{-2}dz^2\).

\begin{teorema}\label{teo:forma-polos-dobles}
	Let \( \varphi=b^2z^{-2}dz^2\) with \( b\neq 0\).
	Let \( \psi:\mathbb{C} \rightarrow \mathbb{C}\setminus \{0\}\) be the covering map:
	\[ z\xmapsto{\psi} e^{z/b}\]
	Then \( \psi\) is a local isometry between the plane and 
	\( \mathbb{C}\setminus \{0\}\) with the metric given by \( \varphi\)
	that carries horizontal lines to horizontal trajectories.
\end{teorema}

Thus, we conclude that if \( b\) is a pure imaginary number,  the foliation is given by concentric circles.
If \( b\) is real, the foliation consists of rays converging to the origin, and all the other cases are logarithmic spirals.
In particular, if \( b\) is not a pure imaginary number, then every horizontal trajectory converges to
the origin in at least one direction.

To end this subsection, we summarize the previous results in the next remarks:
\begin{itemize}
		 \item Around any infinite critical point there is a neighbourhood with the property that any
			 horizontal trajectory that intersects the neighbourhood converges to the critical point
			 in at least one direction.

		\item Infinite critical points are at an infinite distance from the regular points.

		\item Around any finite critical point there is a neighbourhood with the property that exactly 
			\( n+2\) horizontal trajectories converge to the critical point.

		\item Finite critical points are at a finite distance from any regular point.
\end{itemize}

\begin{proof}[Sketch of proof of theorem \ref{teo:criticos}]
	Let's tackle the case when \( p\) is an odd order zero of the quadratic differential.
	Let \( \xi\) be a chart centered at  \( p\). In this chart we have:
	\[ \varphi=\xi^{2k+1}f(\xi)d\xi^2\]
	where \( k\) is a non-negative integer, \( f\) is holomorphic and doesn't vanish at zero.
	
	As a first step, consider the abelian differential
	\[\omega = \sqrt{\varphi}=\xi^{\frac{2k+1}{2}}\sqrt{f(\xi)}d\xi \]
	it is well defined on open sets close enough to zero.

	If \(z\) is a chart centered at \( p\), then
	\[ \varphi=z^*(Cz^{2k+1}dz^2)\]
	if and only if
	\[ \omega=z^*(\pm Kz^{\frac{2k+1}{2}}dz),\]
	where \( K^2=C\), 
	so it is enough to construct a chart
	 \( z\) with this property.

	 It is not difficult to see that finding said chart is the same as solving the complex ordinary differential equation:
	\[ \xi^{\frac{2k+1}{2}}\sqrt{f(\xi)}=\pm Kz(\xi)^{\frac{2k+1}{2}}\frac{dz(\xi)}{d\xi}.\]
	To solve this ODE we can use the power series of  \( \sqrt{f(\xi)}\) 
	and integrate term by term to get:
	\[ \xi^{\frac{2k+3}{2}}g(\xi)=\int\frac{dz(\xi)^\frac{2k+3}{2}}{d\xi}d\xi=z(\xi)^\frac{2k+3}{2},\]
	which in turn gives:
	\[ z(\xi):=\xi g(\xi)^{\frac{2}{2k+3}}\]
	where \( g(\xi)^{\frac{2}{2k+3}}\) is well defined and single-valued in a neighbourhood of the origin.
	This is due to \( g(0)\neq 0\), and the fact that \( z\) is a biholomorphism near the origin because \( \frac{dz(\xi)}{d\xi}(0)=g(0)^{\frac{2}{2k+3}}\neq 0\).

	This concludes the proof of  \ref{teo:criticos} for the case of odd order zeros.
	The case of even order zeros and the case of odd order poles can be handled in a similar fashion.
	The only case slightly different is that of even order poles.
	The problem in this case is that solving the ODE involves
	integrating an expression of the form 
	\[\xi^{-n}\sqrt{f(\xi)} \]
	where logarithmic terms could appear.
	That is the reason why in this case it is only possible to find
	charts where
	\[ \varphi= \left( Cz^k +bz^{-1} \right)^2dz^2.\]
\end{proof}

\subsection{Separating trajectories and decomposition of the surface}

With the previous results, in this subsection we will give a first (very coarse) classification of horizontal trajectories of a quadratic differential.

\begin{definicion}
Let \( \gamma\) be a horizontal trajectory of the quadratic differential  \( \varphi\). We shall say that \( \gamma\) is:
	\begin{itemize}
			 \item \emph{generic}, if in both directions converges to infinite critical points;
			 \item a \emph{saddle} trajectory, if in both directions converges to finite critical points;
			 \item \emph{separating} if in one direction converges to a finite critical point, and in the other to an infinite critical point;
			\item \emph{periodic} if it is diffeomorphic to a circle;
			\item \emph{divergent} in all other cases.
				\footnote{The term \emph{divergent} can be missleading: given that the surface is compact, every trajectory has limit points.
In fact, a trajectory is divergent if in at least one direction has more than one limit point. See \cite[p.~45]{strebel-quadratic}.}

	\end{itemize}
\end{definicion}

Note that no trajectory can be in more than one of the previous categories.

From now on, when we speak of a parametrization of a horizontal trajectory
\( \gamma\) we shall assume that it is a unit speed parametrization, \ie,
\( |\dot{\gamma}|=1\), or equivalently, in a normal chart \(\gamma\) satisfies
\( \gamma(t)-\gamma(s)=\pm|t-s|\).

One of the basic tools that we will use several times is the next proposition,
that allows us to embed any segment of a horizontal trajectory in a
rectangle foliated by horizontal trajectories of the same length:

\begin{proposicion}
Let \( \gamma\) be a horizontal trajectory defined on the interval  \( [a,b]\),
then there is an \( \epsilon >0 \) and a normal chart \( z\) such that:
\( z(\gamma)=[a,b]\subseteq \mathbb{C}\) and 
\( R=\{z\in \mathbb{C}|a<\re(z)<b, |\im(z)|<\epsilon\}\) is a subset of the
image of \( z\).
\end{proposicion}
\begin{proof}[Sketch of proof]
	For every point of \( \gamma\) there is a normal chart
	such that the image of \(\gamma\) is a subset of the real line.
	Hence, there is some
	rectangle of a given height contained in the image of the chart and that is symmetric with respect to the real line.
	By compactness, we can cover \(\gamma\) with a finite number of this charts.
	Thus, there is an \( \epsilon>0\)
	such that any of this charts covers a rectangle with height greater than
	\( \epsilon\).

	Observe that composing with translations and reflections, we can
	modify the charts so that they are compatible on the overlaps.
	This is due to the specific nature of the change of coordinate maps of normal charts.
	In this way we can take the union of the charts. This is a new chart
	with the required properties.

	Another approach to this proposition is to take the exponential mapping
	along the normal directions to \(\gamma\). Since the metric is flat, this is a local isometry.
\end{proof}
\begin{corolario}\label{cor:rectangulo-cubriente}
	If there is a vertical segment of length \( \delta\), whose initial point is  \( \gamma(a)\) and with the property that for any point in the segment,
	the horizontal trajectory starting at that point has a length
	of at least \(b-a\) in both directions, then
	there is a local isometry
	\[ \psi: \{z\in\mathbb{C}|a<\re(z)<b, 0\leq \im(z)<\delta\}\rightarrow X\]
	such that \( \psi([a,b])=\gamma\),
	\( \psi (i[0,\delta])\) is the vertical segment and  \( \psi^{-1}\) is locally a normal chart.
\end{corolario}

Let \( \gamma\) be a divergent trajectory.
Suppose it is defined on a maximal interval:
\[ \gamma: (a,b)\rightarrow X.\]
Because \( \gamma\) is divergent, in at least one direction it does not converge to
a critical point.
Since \( X\) is compact,
\(\gamma\) can be extended indefinitely in this direction.
Without loss of generality, we can assume that \( \gamma\) is defined on 
the interval \( (a,\infty)\).

Let \( a<t_0<t_1<\ldots\) be any increasing sequence of real numbers such that \({ t_n \to \infty}\), then, again by the compactness of  \( X\),
\( \gamma(t_n)\) has at least one accumulation point  \( p\).
The point \( p\) cannot be an infinite critical point because every trajectory close enough to an infinite critical point converges towards it.
By the same reason, if  \(p\) is a finite critical point,
\( \gamma(t_n)\) cannot be on the critical directions.
In this case, every trajectory through the points \( \gamma(t_n)\) is also a part
of  \( \gamma\) hence, at least one of the critical rays are a subset of
the limit points of  \( \gamma\).

In any case, a divergent trajectory has more than one limit point.
This statement can be sharpened as follows:

\begin{proposicion}\label{prop:divergente-recurrente}
	Let \( \gamma\) be a divergent trajectory. Assume that \( \gamma\) is
	parametrized in such a way that the set of limit points of  \( \gamma(t)\) when \( t \to \infty\) has more than one point.
	Let \( B\) be a vertical interval with initial point  \( P_0=\gamma(t_0)\), then, for any point \( P_1=\gamma(t_1)\) after  \( P_0\) (\ie \( t_1>t_0\))
	there is another point \( P_2\) after  \( P_1\) such that \( \gamma\) intersects \( B\) at \( P_2\) 
	in the same direction as  \( P_0\).
\end{proposicion}
\begin{proof}
	Taking the interval
	\( B\) small enough, we can assume that no trajectory starting at a point in \(B\) is periodic nor converges to a finite critical point.

	This guaranties that any trajectory intersecting \(B\) can be extended indefinitely in the same direction as \(\gamma\)

	Corollary \ref{cor:rectangulo-cubriente} allows us to find a local isometry
\[ \psi:\{z\in\mathbb{C} \vert \re(z)\geq t_0, |\im(z)|<\delta\}\rightarrow X\]
where \( \delta\) is than a third of the length of \( B\).

If any of the trajectories starting from \( B\) converges to an infinite critical point, then, since  \( \gamma\) is always at a finite distance of this trajectories,  \( \gamma\) should aslso converge to the critical point, which contradicts the fact that \(\gamma \) is divergent.

This statement can be sharpened as follows: there is an open set  \(  U\)
that is a neighbourhood of all the infinite critical points and such that the image of  \( \psi\) is disjoint from \(  U\). This can be proved using the fact that
the image of \( \psi\) has finite area.

So far we have showed that \( \psi\) is not injective in any subrectangle 
\[{R_t:=\{z\in \mathbb{C} \vert t<\re(z) , |\im(z)|<\delta\} }\]
otherwise, the image should have infinite area.
\footnote{Note that once we have  proved that  \( \psi \) is defined
	for every  \( t>t_0\) and that its image has finite area, the proof is
	essentially the same as Poincare recurrence theorem.}

	Finally, if \( t>t_1\) is such that the image of  the vertical interval at \( t\) under  \( \psi\) meets the original interval, then by the choice of \( \delta\), we can guarantee that  \( \gamma(t)=\psi(t+i0)\) meets \( B\).%
\footnote{For more details on this proof, see the theorem  \( 11.1\) on  \cite[p.~48]{strebel-quadratic}.}
\end{proof}

We have just shown that every divergent trajectory is actually \emph{recurrent}, \ie, it is itself a subset of its limit points.

Together with \ref{cor:rectangulo-cubriente}, the last proposition is one of
the fundamental tools in the classification of neighbourhoods of trajectories:

\begin{proposicion}
	Let \( \gamma\) be a horizontal trajectory. Then:
	\begin{itemize}
			 \item If \( \gamma\) is generic, there is a one-parameter
				 family of generic trajectories around \(\gamma\) \ie, there is a biholomorphism 
				\[ \psi: \{z\in\mathbb{C} \big\vert~ |\im(z)|<\epsilon \}\rightarrow  U\subseteq X\]
				 such that \( \gamma = \psi(\mathbb{R})\) y \( \psi^{-1}\) is a normal chart.
			\item If \( \gamma\) is periodic, in the same way as the last case, there is a  one-parameter
				family of periodic trajectories around \(\gamma\), all of the same length, \ie, there is a covering map, periodic in the horizontal direction 
				\[ \psi: \{z\in \mathbb{C} \big\vert ~|\im(z)|<\epsilon\}\rightarrow  U\]
				 such that \( \gamma = \psi(\mathbb{R})\) and  \( \psi^{-1}\) is locally a normal chart.
			\item If \( \gamma\) is divergent, its closure is a domain
				such that every trajectory in its interior is also divergent, and the boundary is made up of a union of saddle trajectories and zeros of the differential.

	\end{itemize}
\end{proposicion}

Finally, the most important result of this section:

\begin{teorema}\label{teo:descomposicion}
Let \( C\) be the union of every critical point, saddle trajectory and separating trajectory.
Then, every component of  \( X\setminus C\) is isomorphic to one and only one of
the next regions:
\begin{itemize}
		 \item a \emph{horizontal strip}, \ie, the image under an
			 isometry \( \psi\) of a strip 
			 \[ \{z\in\mathbb{C}|0\leq\im(z)\leq h\}\]
			 where the image of the boundary lines is a union of
			 saddle trajectories, separating trajectories, and finite critical points. Also, it can be shown that \( \psi(z\pm t)\) converges to infinite critical points when   \( t\to \infty\);
		\item a \emph{ring domain}, \ie, the image under a local isometry  \( \psi\)
			of a strip
			 \[ \{z\in\mathbb{C}|a\leq\im(z)\leq b\}\]
			 such that \( \psi\) is periodic, with real period.
			 If \( b<\infty\) then the image of the line  \( \{z\in \mathbb{C}|\im(z)=b\}\) is a union of saddle trajectories.
			 Otherwise, \( \psi(z+it)\) converges to a double pole  \( 2\) when \( t\to\infty\). The same statement is true in for  \( a\).
		\item a \emph{half-plane}, \ie, a horizontal strip where we allow
			\({ h=\infty}\). In this case, every trajectory converges to a pole of order grater than  \( 2\);
		\item a \emph{spiral domain}, \ie, the interior of the closure of 
			a divergent trajectory. The boundary of a spiral domain is a union of saddle trajectories.
\end{itemize}
\end{teorema}

\subsection{GMN differentials and the WKB (or Stokes) triangulation}

Despite the fact that the theorem \ref{teo:descomposicion} gives a very complete
description of the global behaviour of horizontal trajectories,
in certain cases this is far from being satisfactory.
This is the case, for example, when there is at least one spiral domain.

In general, there are two possible approaches in the study of quadratic differentials:
Confining attention to the case of quadratic differentials of finite area,
in which case the techniques of ergodic theory become more relevant.
Or focusing on the case of differentials of infinite area (\ie when there is at least an infinite critical point). Under specific circumstances, this case is more
amenable to an algebraic or combinatoric treatment.

In this work we shall take the second approach. More specifically, we will
follow the definitions of the paper \cite{bridgeland2015}.

\begin{definicion}
	A \emph{GMN} differential is a meromorphic quadratic differential such that:
	\begin{itemize}
			 \item every zero is simple,
			\item it has at least one pole,
			\item it has at least one finite critical point.
	\end{itemize}

	Moreover, if a GMN differential has no infinite critical points, we shall say that it has \emph{finite area}, or that it is \emph{finite}.
On the other hand, if it only has infinite critical points, we shall say that
it is \emph{complete}, because the complement of  the critical points is metrically complete.
\end{definicion}

The first occurrence  of this definition is in the paper \cite{neitzke-wall-crossing}.

However, in the majority of this work, we shall confine ourselves to the case
of GMN differentials with double poles only.
If the differentials also don't have saddle trajectories, we will be able to
define the  \emph{WKB triangulation}.%
\footnote{Also, see the book \cite{kawai} where the definition of the WKB triangulation for differentials on the Riemann sphere is given.}

Let us recall the notion of \emph{ideal triangulation},
which is a generalization of the usual notion of triangulation:

\begin{definicion}
	A \emph{punctured surface} is a compact differentiable surface with a finite set of marked points, which we call \emph{punctures}.
An \emph{ideal triangulation} on a punctured surface is a maximal collection of arcs
connecting punctures (possibly the same) and such that 
no two of them are homotopic relative to their boundary, and neither intersect one another. 
Two triangulations are equivalent if their arcs are isotopic relative to the boundary.
\end{definicion}

\begin{teorema}\label{teo:triangulacion-wkb} 
	Let \( \varphi\) be a GMN differential on \( X\) with only double poles and
	no saddle trajectories.
	Then, taking as punctures the set of poles and as arcs, one generic trajectory of every horizontal strip, we get an ideal triangulation of  \( X\)
	which we call the \emph{WKB triangulation}.
\end{teorema}
\begin{proof}
	First we note that by the definition of a GMN differential
	there is at least one simple zero and a pole, and this in turn implies that
	there is at least one horizontal strip.
	
	Also, the conditions we have imposed and theorem
	 \ref{teo:descomposicion} allows us to conclude that the only
	 possible regions that occur in the decomposition of the surface are horizontal strips.

	It is clear that the arcs, that in this case are generic trajectories,
	are disjoint and connect punctures.
	We need to check that there are no homotopic arcs and that the collection
	of arcs is maximal.

	If two generic trajectories that don't belong to the same
	horizontal strip are homotopic, then they are the boundary of a
	contractible set that doesn't have any pole of the differential.
	Since they don't belong to the same horizontal strip,
	there is at least one simple zero in the contractible set that they bound.
	Every one of the three rays coming from the zero should converge to one
	of the limit points of the generic trajectories, so at least two of the rays converge to the same pole.

	Call \( R_1\) the region bounded by those two rays, the pole and the zero.
	The set 
	\( R_1\) is also a contractible set without poles in the interior.
	If \( R_1\) also doesn't have zeros on the interior, it should be isomorphic
	to a half-plane, but we known that the limit point in any half-plane is a pole of order grater than  \( 2\).

	On the other hand, if there is a zero on  \( R_1\) two of the rays
	coming from the zero converge to the pole, and consequently they bound a new
	region  \( R_2\) with the same properties as  \( R_1\).

	Thus, because there are no half-planes, we could continue this process indefinitely, but there are only a finite amount of zeros so we arrive at a contradiction.

	This argument proves that the arcs are not homotopic.
	We shall prove that the collection of arcs is maximal in the next section
	after we develop the necessary tools
	to compute the number of horizontal strips of a given differential.
\footnote{A slightly different and more general proof can be found in the lemma  \( 10.1\) of \cite{bridgeland2015}.}
\end{proof}

\begin{figure}
\centering
\includegraphics{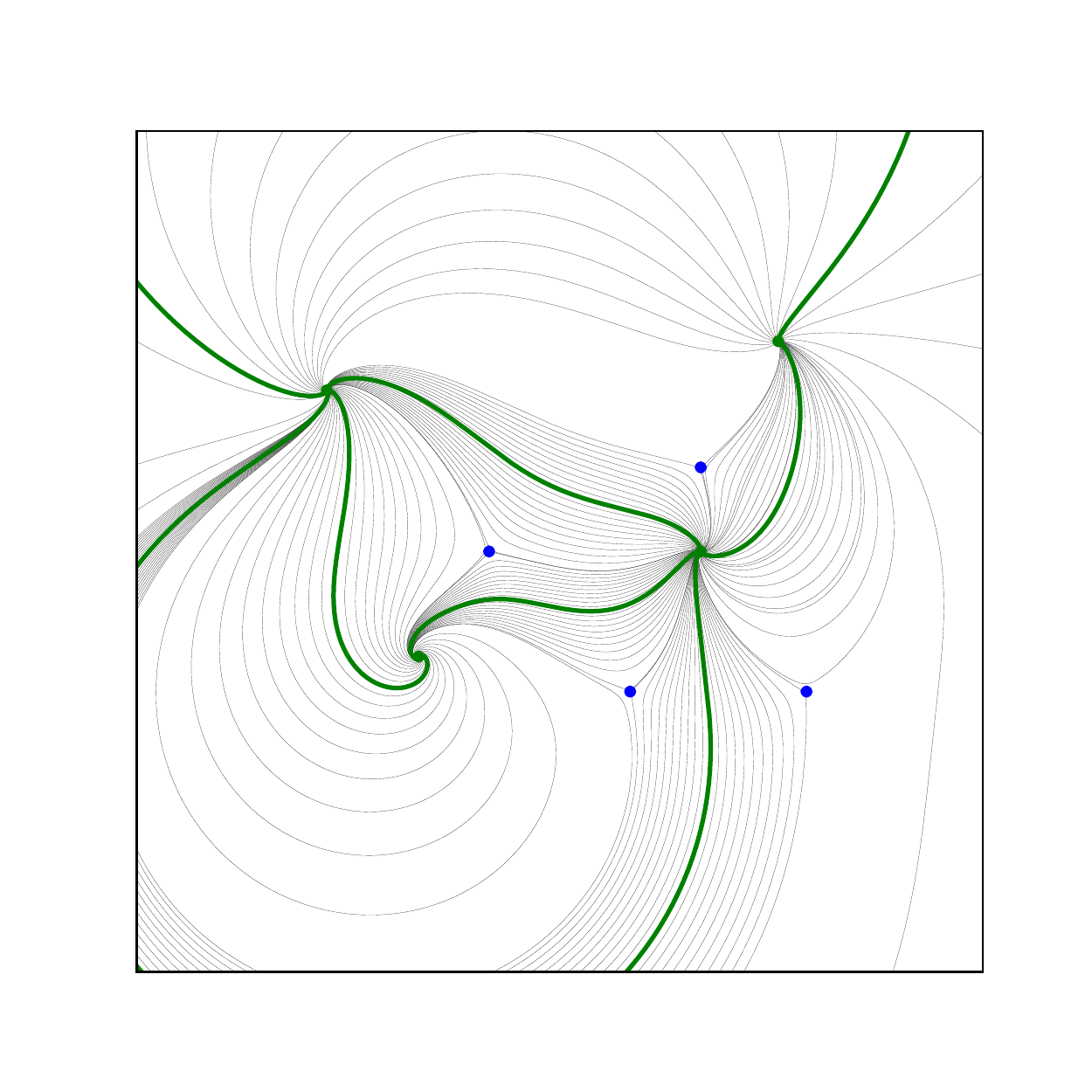}
\caption{Horizontal trajectories of a quadratic differential on the complex plane,
	with four double poles and four simple zeros.
	The green lines are arcs of the WKB triangulation.
}\label{fig:triangulacion}
\end{figure}

%% file: spectral.tex
\section{Spectral covers, homology and periods}%
\label{cap:espectral-hom}
The goal of this section is to give a homological description of 
a quadratic differential.
This is analogous to the description of Abelian differentials by their periods \ie the integral of the Abelian differential along a basis for the homology of the surface.

The first step is a construction needed to be able to define
the ``square root'' of a quadratic differential.
As usual this is done by taking a two fold branched covering
of the surface.
This covering is called the \emph{spectral cover}.
The details of its construction will be given in section \ref{sec:cubierta-espectral}.

In a natural way, the ``square root'' of the quadratic differential
is a meromorphic \( 1-\)form on the spectral cover, and the periods
can be defined by integrating this form over the elements of a
basis of the first homology group of the covering.%
\footnote{More precisely, over the \emph{hat homology group} which is a subgroup of the first homology group of the cover.}
Thus, it is important to compute the genus and the first Betti number of
the spectral cover. Necessary tools for this will be given in the first section while a detailed study of the homology of the spectral cover will be carried in the last section.

Most of the results and definitions of this section are taken from \cite{bridgeland2015}, except those of the first section which are classical results in the theory of Riemann surfaces and line bundles.
Two possible references for this last topic are 
\cite{donaldson2011riemann} and
\cite{narasimhan1992compact}.

\subsection{Existence of quadratic differentials}

In the previous section we developed a partially satisfactory
understanding of the geometry of a quadratic differential.
We employed a local classification of the geometry of the differential
together with some global arguments, which eventually led us to
Theorem \ref{teo:descomposicion}.

In order to deepen those results, we shall need a language
to be able to speak of families of quadratic differentials.
In particular, we will also establish the existence of
quadratic differentials on any Riemann surface.

\subsection{Digression: divisors, Riemann-Roch Theorem and the Picard group}

One way to systematically handle zeros and poles of a holomorphic section of a line bundle is to use \emph{divisors}:

\begin{definicion}
	A \emph{divisor} \( D\) on \( X\) is a finite formal sum of points of \(X\).
	We usually write \[ D=\sum n_p p\]
	where \( n_p\) is an integer and only a finite amount of this numbers is different from zero.
	The \emph{degree} of a divisor \( D=\sum n_p p\) is the number \( \grad(D)= \sum n_p \in \mathbb{Z}\).
\end{definicion} 

The divisors on a given Riemann surface form an Abelian group under
pointwise addition.
We will write \( \Divisor(X)\) to denote this group. The degree is a
group homomorphism 
	\[ \grad: \Divisor(X)\rightarrow\mathbb{Z}.\]

	If \(\sigma\) is a meromorphic section of a line bundle  \( L\) with  zeroes  \( p_1,\ldots,p_n\)
and poles \( q_1,\ldots,q_m\), then the divisor:
\[ d(\sigma):= \sum_{i=1}^{n} \ord_{p_i}(\sigma)p_i + \sum_{i=1}^{m} \ord_{q_i}(\sigma)q_i \]
where \( \ord\) is the order of  \( \sigma\) at a zero or a pole, is called the \emph{divisor of} \( \sigma\).
We will simply call \( \grad(d(\sigma))\) the degree of \( \sigma\).

It is possible to define a partial ordering on the divisor group
by declaring
\( D=\sum n_p p\geq D^\prime=\sum m_p p\) if and only if \( n_p\geq m_p\) for every \( p\in X\).

For example, if follows that \( d(f)\geq p-q\) is the same as saying that  \( f\) has at least a simple zero at  \( p\) and at most a simple pole at  \( q\).

We shall write \( \mathcal{O}_X\) for the sheaf of holomorphic
functions on a Riemann surface \( X\).
If we are working on a fixed surface we will just write  \( \mathcal{O}\).

If \( D\) is a divisor and \(L \) is a line bundle, we can define a
sheaf of 
\(\mathcal{O} \)-modules by the rule:
\[ \mathcal{L}_D (\mathcal{U}):=\left\{\,\sigma~ \middle| \, \sigma:\mathcal{U}\rightarrow L \, \textrm{is a meromorphic section and } d(\sigma)\geq D \vphantom{\frac12}\right\}.\]
When \( D=0\) we will write \( \mathcal{L}\) instead of \( \mathcal{L}_D\).
Also, by \( H^i(L_D)\) we shall mean the sheaf cohomology groups of \( \mathcal{L}_D\).

It is well known that \( H^0(L)\) is canonically isomorphic to the
vector space of global sections of  \( L\)
and that the higher cohomology groups can be defined more or less explicitly with \emph{\v{C}ech cohomology} or more abstractly 
as \emph{derived functors} of the global sections functor.

One of the most important results about sheaf cohomology groups of line bundles is the so called \emph{finiteness theorem};

\begin{teorema}[finiteness]\label{teo:finitud} 
	Let \( L\) be a line bundle on  \( X\).
	The cohomolgy groups of \( L\) satisfy:
	\[ H^i(L)=0 \quad\textrm{ if }\quad i>1\]
	\[\dim_\mathbb{C}H^0(L)<\infty \text{~~ and~~ }\dim_\mathbb{C}H^1(L)<\infty.\]
\end{teorema}

With this theorem we can easily proof the existence of meromorphic sections for any line bundle:
\begin{corolario}
	Let \( L\) be a line bundle on  \( X\) and \( p\in X\).
	There is a meromorphic section \( \sigma\) of \( L\) that is holomorphic on  \( X\setminus \{p\}\).
	
\end{corolario}

Let \( L\) be a line bundle and  \( d(\sigma)\) be the divisor of any section \( \sigma\) of \(L\) .
By the previous corollary, there is at least one such section.
On the other hand, if  \( \sigma\) and  \( \tilde\sigma\) are two meromorphic sections of 
\( L\) the quotient \( \sigma/\tilde\sigma\) is a well defined meromorphic function on \( X\) so  \( d(\sigma)=d(\tilde\sigma)+d(\sigma/\tilde\sigma)\).

Divisors of meromorphic functions are called \emph{principal divisors}
and form an additive subgroup of \( \Divisor(X)\).
Let us write  \( \PDivisor(x)\) for the subgroup of all principal divisors.
The quotient \(\Divisor(X)/\PDivisor(X)\) is called the \emph{divisor class group} of \( X\).

With this in mind we can prove the next theorem:

\begin{teorema}\label{teo:picard-divisores} 
	The function that maps a line bundle \( L\) on \( X\)
	to the divisor class of any meromorphic section of \( L\)
	is a group isomorphism:
	\begin{align*}
		\pic(X)& \xrightarrow{\Divisor} \Divisor(X)/\PDivisor(X)\\
	L &\mapsto [d(\sigma)]
\end{align*}
\end{teorema}

Recall that for any meromorphic functions, the number of
zeros and poles is the same (counting multiplicity).

This statement can be rephrased by saying that the degree of any principal divisor is zero.
Thus, the degree can be unambiguously defined for any divisor class in
the divisor class group.

So, in particular, a consequence of the previous theorem is that
any line bundle  \( L\) has a well defined \emph{degree} 
\( \grad(L)\) that it can be computed as the degree of \emph{any}
meromorphic section of \(L\).

Finally we can state two of the main theorems in the theory of Riemann
surfaces: the Riemann-Roch theorem and the Serre duality theorem.

\begin{teorema}[Riemann-Roch]\label{teo:riemann-roch}

	Let \(L\) be a line bundle on \(X\). Then 

\[\dimension_\mathbb{C}\H^0(L)-\dimension_\mathbb{C}\H^1(L) = \grad(L) - g +1\]
where \( g\) is the genus of \( X\).

\end{teorema}

\begin{teorema}[Serre duality]\label{teo:dualidad-serre}
	Let \(L\) be a line bundle on \(X\). Then 
\[\H^0(L) \cong \H^1(L^*\otimes K_X)^*\]
or, equivalently:
\[\H^0(L^*\otimes K_X) \cong \H^1(L)^*\]
\end{teorema}

This last result, due to Jean Pierre Serre, relates the first cohomolgy groups to the zeroth cohomology groups.
It is the main reason why the theory originally developed without
reference to the first cohomology groups:
everything about them can be recast (albeit in a possibly obscure way) in terms of meromorphic sections
of a related line bundle.%
\footnote{For a discussion of this fact see \cite[p.~186]{donaldson2011riemann}.}

As a first application of this theorems, we shall compute the cohomology groups and degree of two of the most important line bundles:
the trivial and canonical bundles.

First we note that by the Serre duality theorem:
\[1=\dim \H^0(\mathbb{C}_X) =\dim \H^1(K_X) \]
\[\dim \H^0(K_X) =\dim \H^1(\mathbb{C}_X) \]
On the other hand, by the Riemann-Roch theorem:
\[\dimension\H^0(\mathbb{C}_X)-\dimension\H^1(\mathbb{C}_X) = \grad(\mathbb{C}_X) - g +1\]
which combined with the previous assertion and the fact that
\( \grad(\mathbb{C}_X)=0\) is the same as
\[ 1-\H^0(K_X)=-g+1\]
so we conclude that \( \H^0(K_X)=g\).

Again, by the Riemann-Roch theorem, but now for the canonical bundle, we have:
\[\dimension\H^0(K_X)-\dimension\H^1(K_X) = \grad(K_X) - g +1\]
so 
\[ g-1=\grad(K_X) -g +1\]
which leads to:
\begin{corolario} The degree of the canonical bundle of a compact Riemann surface \( X\) is \( 2g- 2\) where \( g\) is the genus of the surface:
	\[ \grad(K_X)=2g-2.\]
\end{corolario}

\vspace{1cm}
\centerline{\adforn{49} } 
\vspace{1cm}

With this in hand we will easily compute the dimension of the space
of \emph{holomorphic} quadratic differentials on a compact Riemann surface \( X\):

Since the degree is a group morphism, \( \grad(K_X^*)=-\grad(K_X)=2-2g\).

There are three different cases:

If \( g>1\), \( \grad(K_X^*)<0\) so \( K_X^*\) has no holomorphic sections. Thus  \( \H^0(K_X^*)=0\).
By the duality theorem,
\[ \dim\H^0(Q_X)=\dim\H^0(K_X\otimes K_X)=\dim\H^1(K_X^*)\]
so, to compute this last number, we can resort to the the Riemann-Roch  theorem:
\[ \dim\H^0(K_X^*)-\dim\H^1(K_X^*)=\grad(K_X^*)-g+1=3-3g\]
since \( \H^0(K^*_X)=0\) we conclude that
\( \dimension\H^0(Q_X)=3g-3\).

If \( g=1\), then \( \dimension\H^0(K_X) =1\), so there is at least
one holomorphic differential not identically zero.
On the other hand,  \( \grad(K_X)= 2g-2=0\), \ie, any holomorphic
differential has vanishing divisor, in particular it doesn't have zeros.
This last assertion is equivalent to
 \( K_X\) being a trivial bundle, so  \( Q_X=K_X\otimes K_X\) is also trivial.

The only case remaining is \( g=0\):
in this case, \( \grad(K_X)= 2g-2 = -2\) and \( \grad(Q_X) = -4\), so \( Q_X\)
has no holomorphic sections.

Let's remark that in any case, since the degree of the canonical bundle is \( 2g-2\), the degree of \( Q_X\) is  \( 4g-4\).

\subsection{The spectral cover}%
\label{sec:cubierta-espectral}

In this section we define the \emph{spectral cover} for a given quadratic differential.differential.

The spectral cover%
\footnote{We took the name from \cite{bridgeland2015}. In other situations,
	for example in the study of flat surfaces, it is called the \emph{orientation cover}, because pulling back the foliation to this cover gives an orientable
	foliation.}
	is a construction that  allows us to take a ``square root'' of
	the quadratic differential.

More precisely, the spectral cover is a two fold branched covering of the
original surface \( X\), with the property that the pullback of the quadratic
differential splits as the tensor square of a global meromorphic \( 1-\)form.

This construction is relevant because it allows us to define
the periods of the quadratic differential.
The periods are used in a similar fashion as in the theory of
Abelian differentials where they are used to
parametrize certain  
families of differentials.%
\footnote{For a brief exposition on \emph{period coordinates} in the case of
Abelian differentials, see
 \cite{wright2014translation}.}

Let \( \varphi\) be a meromorphic quadratic differential on \( X\).

Naïvely we could try to define the spectral cover as the set:

\[ X_\varphi:=\{\alpha \in K_x| \alpha \otimes \alpha= \varphi(x)\}	\subseteq K_X\]
and use the restriction of the projection of \( K_X\) to \( X\) as
the covering map.

In this case it is natural to consider the ``tautological'' meromorphic
 \( 1-\)form
that could be defined as%
\footnote{This expression is slightly wrong, but we shall refrain from
	giving more details until we develop the definitive version of the spectral cover.}
\begin{align*}
		X_\varphi & \xrightarrow{\omega} K_X \\
		\alpha &\mapsto \alpha
\end{align*}
and where clearly \( \omega\otimes\omega=\varphi\).

The problem with this construction is that locally  \( X_\varphi\) is
isomorphic to a possibly singular algebraic variety, or is not defined at all.

Specifically, if \( x\) is an order \( n\) zero of \( \varphi\)
it is easy to see that on a neighbourhood of \( x\), \( X_\varphi\) is
isomorphic to the algebraic variety in \( \mathbb{C}^2\) given by the equation
 \( z^n=w^2\). This variety has a singularity at the origin if  \( n>1\).
Even worse, if  \( \varphi\) has a pole at  \( x\), then \( X_\varphi\) is not even
defined over \( x\).

Because of this, we will need to twist the bundle  \( Q_X\) together with the section
\( \varphi\)
and for this we will review theorem \ref{teo:picard-divisores} in greater detail.

Let \( D\) be a divisor on \( X\). By theorem \ref{teo:picard-divisores}
we know there is a line bundle \( L_D\) such that
if \( \sigma\) is a meromorphic section of  \( L_D\) then \( d(\sigma)\)
is in the same class as  \( D\).
Let \( \sigma\) be any meromorphic section of  \( L_D\)
(by theorem \ref{teo:finitud} there is at least one such section).
Then, by the previous remark, there is a meromorphic function 
\( f\) such that \( D=d(\sigma) +d(f)\).
Therefore the section \(\sigma_D:= f\sigma\) has divisor equal to  \( D\).

Let us remark that the property \( d(\sigma_D)=D\) uniquely determines
the section \( \sigma_D\) up to a non zero scalar multiple.

To twist the bundle \( Q_X\) and solve the problems discussed early on, the idea
is to take the tensor product with a bundle \( L_D\)
in such a way as to make the section \( \varphi\otimes \sigma_D\) to only
have simple zeros.

Let \( D_\varphi= d(\varphi)\) be the divisor of the quadratic differential  \( \varphi\).

If \( D_\varphi=\sum_{i=0}^n n_i p_i- \sum_{i=0}^m m_i q_i\), where the \( n_i\)
and the \( m_i\) are positive integers, then define the divisor 
\( D:=\sum_{i=0}^m 2 \left \lceil \frac{m_i}2 \right \rceil  q_i -\sum_{i=0}^n 2\left \lfloor \frac{n_i}2 \right \rfloor  p_i \).
Clearly \( D_\varphi + D\) is zero at every regular point and every even order critical point, and is exactly one at odd order critical points.

Consider the bundle \( Q_\varphi:= Q_X\otimes L_D\).
The tensor product \( \varphi\otimes \sigma_D\) is a section of  \( Q_\varphi\)
with divisor equal to \( D_\varphi +D\), so it is a holomorphic section 
with simple zeroes at the even order critical points of  \( \varphi\) and no
zeros of higher order.

We can also define the divisor
\( D/2:=\sum_{i=0}^m  \left \lceil \frac{m_i}2 \right \rceil  q_i -\sum_{i=0}^n \left \lfloor \frac{n_i}2 \right \rfloor  p_i\)
then \( 2(D/2)=D\) so that \( L_{D/2}\otimes L_{D/2} = L_D\).%
\footnote{Note that the divisor  \( D\) is the only ``even'' divisor such that  \( \varphi\otimes\sigma_D\) is holomorphic and doesn't have higher order zeros.}
Define \( K_\varphi:= K_X \otimes L_{D/2}\). Then  \( K_\varphi \otimes K_\varphi = Q_\varphi\).

With this definitions we can finally define the \emph{spectral cover}:

\begin{definicion}
The \emph{spectral cover} of the quadratic differential \( \varphi\)
is the set
\[ X_\varphi := \{(x,\alpha)\in K_\varphi | \alpha \otimes\alpha = \varphi(x)\otimes \sigma_D(x) \}\subseteq K_\varphi.\]
Let \( \pi\) denote the restriction of the projection of  \( K_\varphi\) over \( X\) to \(X_\varphi\).
\end{definicion}

\begin{proposicion}\label{prop:propiedades-cubierta-espectral}
	Let \( X_\varphi\) be the spectral cover of  \( \varphi\). Then:
	\begin{enumerate}
			 \item \( X_\varphi\) is a Riemann surface.
			\item The function \( \pi\) holomorphic and
				a two-sheeted branched covering of \( X\).
	The branch points of  \( \pi\) are exactly the zeros an poles of odd order of  \( \varphi\).
	\item There is a  meromorphic \( 1-\)form \( \omega\) on \( X_\varphi\) such that \( \omega\otimes\omega= \pi^*(\varphi)\).
	\end{enumerate}

\end{proposicion}
\begin{proof}
	Consider a trivialization of  \( K_\varphi\) over an open set \( U\).
	This trivialization gives rise to a trivialization of the bundle
	\( Q_\varphi\), and in this trivialization
	\( \varphi\otimes \sigma_D\) corresponds to a holomorphic function  \( f\) on \( U\).
	Thus \( \pi^{-1}(U)\) is equivalent to the set
	\[ \{(x,z)\in U\times \mathbb{C}| z^2=f(x)\}\subseteq U\times \mathbb{C}.\]
	Since the function  \( f\) is holomorphic and only has simple zeros, 
	the set \( \pi^{-1}(U)\) is smooth.
	
	The trivialization was arbitrary, thus the assertion above is valid
	in a neighbourhood of every point of  \( X_\varphi\),
	so it is a Riemann surface.

	In any of this trivializations, the function  \( \pi\) is equivalent to the projection of \( U\times \mathbb{C}\) to the first coordinate.
	This proves the second assertion.

We can choose the section  \( \sigma_{D/2}\)
in such a way that \( \sigma_{D/2}\otimes \sigma_{D/2}=\sigma_D\) and that
if \( x\in X\) and \( \alpha\in (K_\varphi)_x\) then \( \frac\alpha{\sigma_{D/2}(x)}\in (K_X)_x\).

	Hence \( \frac\alpha{\sigma_{D/2}(x)}\) is a linear functional on  \( \Tangente_xX\).

	We shall define \( \omega\) describing its action on tangent vectors:

	If \( v\in \Tangente_{(x,\alpha)}X_\varphi\) define
	\[ \omega{(x,\alpha)}(v):= \frac\alpha{\sigma_{D/2}(x)}\left[\Tangente_{(x,\alpha)}\pi(v)\right]\]
	Since \( \Tangente_{(x,\alpha)}\pi\) is a linear transformation and \( \frac\alpha{\sigma_{D/2}(x)}\)
	is also linear, \( \omega{(x,\alpha)}\) is a linear functional on \( \Tangente_{(x,\alpha)}X_\varphi\)
	so that \( \omega\) is a meromorphic
	\( 1-\)form on \( X_\varphi\). We will call \(\omega\) the tautological
	form.

	To compare the differentials \( \omega\otimes \omega\) and  \(\pi^*(\varphi)\) it is enough to compute their value on a vector of the form \(v\otimes v\).
	Let \( v\in \Tangente_{(x,\alpha)}X_\varphi\) then
	\( \omega\otimes \omega(v\otimes v)=\frac{\alpha\otimes \alpha}{\sigma_D(x)}[\Tangente_{(x,\alpha)}\pi(v) \otimes \Tangente_{(x,\alpha)}\pi(v)] \)
	but since \( (x,\alpha)\in X_\varphi\) we have that \( \frac{\alpha\otimes \alpha}{\sigma_D(x)}=\frac{\varphi(x)\otimes \sigma_D(x)}{\sigma_D(x)}=\varphi(x) \) thus \( \omega\otimes \omega(v\otimes v)= \varphi(x)[\Tangente_{(x,\alpha)}\pi(v) \otimes \Tangente_{(x,\alpha)}\pi(v)]\).
	We have just shown that \( \omega\otimes \omega= \pi^*(\varphi)\).

\end{proof}

Fiberwise multiplication by  \( -1\), is an automorphism of the bundle  \( K_\varphi\). It restricts to a holomorphic involution of  \( X_\varphi\).
We denote this involution by \( \tau\); \( \tau\) preserves the fibers of \( \pi\)
and \( \omega\) is  \emph{anti-invariant} under \( \tau\), \ie \( \tau^*(\omega)=-\omega\).

There exists a unique meromorphic vector field  \( V_\omega\) satisfaying
\( \omega(V_\omega)=1\),
\ie, which is dual to the \( 1-\)form \( \omega\).
It is clear that \( V_\omega\) is tangent to the foliation of the quadratic differential \( \omega\otimes\omega\),
so it provides an orientation for its leaves.
The vector field \( V_\omega\) is anti-invariant under \( \tau\).

From this facts we can conclude that \( \pi\) maps integral trajectories of  \( V_\omega\) to horizontal trajectories of \( \varphi\);
the foliation of  \( \varphi\) is orientable if and only if
\( X_\varphi\) has two components. Each of the two points on a generic fiber of  \( \pi\) correspond to each of the two possible orientations of the 
foliation and \( \tau\) swaps the two orientations.

If we consider the metric of the quadratic differential 
\( \omega\otimes \omega\), \( \pi\) is an isometry in the complement of the
branch points.

Since \( X_\varphi\) is a branched covering of two sheets,
\( \pi\) is degree two, so is locally equivalent to the transformation
\( z\mapsto z^2\).
If \( x\in X\) is a critical point of odd order of  \( \varphi\)
and we write the differential in a local chart \( \xi\),
then we have
\( \varphi= \xi^{2n+1}d\xi^2\), so locally \( \pi(z)=\xi(z)=z^2\),
and after taking the pullback of \(\varphi\) under \(\pi\) we conclude that:
\[ \pi^*(\varphi)=\xi(z)^{2n+1}d\xi(z)^2=z^{4n+2}2z^2dz^2=2z^{4n+4}dz^2.\]
Among other things, this implies that the fiber over a simple pole
is one regular point of \( \pi^*(\varphi)\).
In particular, if the differential \( \varphi\) only has simple poles and zeroes,
the differential \( \omega\) is holomorphic,
\ie it is an Abelian differential.

\subsection{Homology of the spectral cover}

In this section we will define the \emph{hat homology}.
This will allow us to see the quadratic differential as a
cohomology class, and this in turn will be necesary to define
the periods of the quadratic differential.

We shall work with a fixed quadratic differential
\( \varphi\) for the rest of the section.

Let \(\widehat{X} \) be the spectral cover of \( \varphi\).
Let \( \omega\) be the meromorphic tautological \( 1-\)form.

We let \( \CritInf(\omega)\) denote the set of critical points of \( \omega\)
 and \( \widehat X^\circ\) denote the complement .
Since \( \omega\) is holomorphic on \( \widehat X^\circ\),
it is a closed form and hence defines a de Rham cohomology class.

If \( [\alpha]\in \H_1(\widehat X)\) is the homology class of a \( 1-\)chain \( \alpha\), we can integrate \( \omega\) along \( \alpha\)
and the result only depends on the homology class.

The involution  \( \tau\) of \( \widehat X\) induces a linear action on \( \H_1(\widehat X)\)
that is also an involution, so the homology group splits as the direct sum
of the \( \tau-\)invariant  part and the \( \tau-\)anti-invariant part:
\[ \H_1(\widehat X)= \H_1(\widehat X)^+ \!\oplus \H_1(\widehat X)^-\]
If \( [\alpha]\) is an invariant class, then:
\[ \int_\alpha \omega = \int_{\tau(\alpha)}\omega= \int_\alpha \tau^*(\omega)= -\int_\alpha \omega\]
so the integral vanishes.
So we are led to consider only anti-invariant homology classes.

\begin{definicion}
	The \emph{hat homology group} of the differential \( \varphi\) is the \( \tau-\)anti-invariant part of the first homology group of \( \widehat X^\circ\).
	We write \( \widehat{\H}(\varphi)\) for this group.
\end{definicion}

As we have said before, the quadratic differential \( \varphi\) gives rise to
a group homomorphism that we shall denote \( \widehat{\varphi}\):
\begin{align*}
\widehat{\varphi}:\widehat{\H}(\varphi)&\xrightarrow{~~~~~~~} \mathbb{C}\\
	 [\alpha] &\mapsto  \frac12\int_\alpha \omega.
\end{align*}

Call this homomorphism the \emph{period} of \( \varphi\).\label{pag:periodo}
We shall also call \( \widehat{\varphi}([\alpha])\) the \( \varphi-\)period of \( \alpha\), or just the period of  \( \alpha\).

It is clear that \( \Hom_\mathbb{Z}(\widehat{\H}(\varphi),\mathbb{C})\cong \mathbb{C}^N\) where \( N\) is the rank of the hat homology group of  \( \varphi\).

Later we shall describe a way to revert this construction and get a quadratic differential from a group morphism. So it is important to compute the rank of the hat homology group.

As a first step let us find the genus of \(\widehat X\) in terms of the genus of  \( X\) and the order of the zeros and poles of  \( \varphi\).
For this we use the \emph{Riemann-Hurwitz formula} applied to the projection  \( \pi\).
Since \( \pi\) is a two sheeted branched covering, it has degree two, the ramification points have multiplicity two and al the other points have multiplicity one.

With this in mind, the formula reads:
\[ 2\hat g -2 = 2(2g -2) + \#\{\textrm{ramification points}\}\]
where \( \hat g\) is the genus of \( \widehat X\) and \( g\) is the genus of  \( X\).
By the construction of the spectral cover, it is clear that the number of
ramification points is the amount of odd order critical points.
In general, it is hard to say anything about this number, but in the case of a
GMN differential, since they only have simple zeros, we can easily compute the
number of these:

Using the notation of the last section,
the divisor of the quadratic differential is 
\( D_\varphi=\sum_{i=0}^{n} n_i p_i- \sum_{i=0}^m m_i q_i\) where \( n_i=1\) for
every \( i\),
so that
\[4g-4= \grad(Q_X) = n-\sum_{i=0}^m m_i \]
\ie 
\(n= 4g-4+\sum_{i=0}^m m_i \). Let \( m_p\) be the number of even order poles.
Then:
\[ 2\hat g  = 4g -2 + (4g-4+\sum_{i=0}^m m_i)+ (m-m_p)\]
which is the rank of \( \H_1(\widehat X)\).

Recall that the simple poles of \( \varphi\) become regular points
on the spectral cover, and that every even order pole have
exactly two preimages on  \( \widehat X\).

Applying this to the surface \( \widehat X\) and given that \( \CritInf(\omega)\) 
is the preimage of the infinite critical points of \( \varphi\),
we conclude that  \( \CritInf(\omega)\) has exactly \( m+m_p -s\) points, where \( s\) is the number of simple poles of \( \varphi\).

Now, removing  \( k\) points of a compact surface, increases the rank of the first
homology group by  \( k-1\).
This is a classic result in algebraic topology that can be proved, for example,
with the Mayer-Vietoris exact sequence.
With this in mind, we have:
\[ \rango \H_1(\widehat X ^\circ)= \rango \H_1(\widehat X) +(m+m_p -s-1)= 8g-6+\sum_{i=0}^m m_i+2m-s-1.\]

Let \( X^\circ\) be the complement of the infinite critical points of  \( \varphi\).
Then
\[ \rango\H_1(X^\circ)= \rango\H_1(X) +m-s-1=2g+m-s-1. \]

It is clear that \( X^\circ\) is the quotient under \( \tau\) of \( \widehat X^\circ\).
In general, for finite group actions we have:%
\footnote{This assertion is false for integer coefficients, hence the need to take
	rational coefficients.
	(Actually, it is enough to take a coefficient ring in which the order
	of the group is an invertible element)
	See \cite{bredon1972introduction}, Theorem \( 2.4\) p. \( 120\).}
\[ \H_\bullet(X,\mathbb{Q})^G \cong \H_\bullet(X/G,\mathbb{Q})\]
so in this case, we can conclude that
\( \rango \H_1(X^\circ)= \rango \H_1(\widehat X^\circ)^+\).

Finally, putting everything together, we have:
\begin{align*}
	\rango\widehat\H(\varphi)=\rango \H_1(\widehat X^\circ)^-=&\rango \H_1(\widehat X^\circ) -\rango \H_1(X^\circ)\\
	=& 8g-6+\sum_{i=0}^m m_i+2m-s-1 -(2g+m-s-1)\\
        = & 6g -6 + 	\sum_{i=0}^m m_i+m.
\end{align*}

Thus we have proved:
\begin{proposicion}\label{prop:dimension-h-gorro}
	The rank of the hat homology group of a GMN differential is:
	\[ N(\varphi):=6g -6 + 	\sum_{i=0}^m m_i+m\]
	where \( m\) is the number of poles of \( \varphi\) and \( m_i\) is the order of the \( i-\)th pole.

\end{proposicion}
This Proposition corresponds to Lemma 2.2 of \cite{bridgeland2015},
and the proof is loosely based on that of the lemma.

Before ending this section, let's remark that the intersection form of
\( \H_1(\widehat X^\circ)\) restricted to  \( \widehat \H(\varphi)\) is a
degenerate bilinear form.

Later we shall need another bilinear pairing that is not degenerate.
Recall that one way to define the intersection pairing is combining the
isomorphism coming from Poincare duality and the one coming from the
universal coefficient theorem.

In our situation, instead of using Poincare duality, we can use a slightly
more general version of the duality theorem:

\begin{teorema*}[Lefschetz duality]
Let \( X\) be an oriented \( n\)-dimensional compact manifold with boundary.
There is an isomorphism
\[ D_L:\H^\bullet(X, \partial X) \rightarrow H_{n-\bullet}(X\setminus \partial X)\]
\end{teorema*}

In the case that \( X\) is a surface, we get an isomorphism between
\(\H^1(X, \partial X)\) and \({ H_1(X\setminus \partial X)}\).
Besides, for the case we are interested in, the surfaces have torsion-free homology
so from the universal coefficient theorem we conclude that:
\[ \H^\bullet(X,\mathbb{Z})\cong \Hom_\mathbb{Z}(\H_\bullet(X,\mathbb{Z}),\mathbb{Z}).\]
Combining both isomorphisms we can define a non-degenerate bilinear pairing:
\[ \langle ~~,~~\rangle: \H_1(X, \partial X)\times H_1(X\setminus \partial X)\rightarrow \mathbb{Z}\]

The surface \( \widehat X^\circ\) is not a compact surface with boundary, however,
for every point of \( \CritInf(\omega)\) we can remove a small open disc around that point.
Let \( \widetilde X\) be the resulting surface. Then it is clear that  \( \widehat X^\circ\) and
\( \widetilde X\) are homotopically equivalent and that \( \widetilde X\) is a compact surface with boundary.

Finally, using homotopy invariance of homology, and the excision axiom
we can define an isomorphism
\( \H_1(\widetilde X, \partial \widetilde X)\cong\H_1(\widehat X, \CritInf(\omega))\)
that combined with the bilinear pairing previously defined gives:
\[ \langle ~~,~~\rangle: \H_1(\widehat X, \CritInf)\times H_1(\widehat X^\circ)\rightarrow \mathbb{Z}\]

As in the case of the usual intersection pairing, this bilinear form can be
computed as the oriented intersection number of transversal curves representing the
homology classes.
We shall use this fact later when we define the  \emph{standar saddle classes} which
in turn can be used to define a basis for the hat homology group of the surface.

%% file: strips.tex
\section{Horizontal strips and periods}
\label{sec:franjas}

In this section we will focus our study to the case of quadratic differentials whose decomposition consists only of horizontal strips and half-planes.

We shall define a natural basis for the hat homology group of those
quadratic differentials and explicitly describe the set of
isomorphism classes of horizontal strips.

\subsection{Horizontal strips and saddle classes}

The goal of this section is to find a family of standard models
for the horizontal strips.

\subsubsection{Trajectories}

So far, in the study of quadratic differentials, we have focused only on the horizontal foliation. However, there is a whole family
of foliations parametrized by \( \mathbb{S}^1\cong \mathbb{R} /\mathbb{Z}\).

To see this, recall that those structures that are invariant under
normal coordinate change maps are intrinsic to the differential.
In this way we defined the metric and horizontal foliation.

Locally, these structures correspond to structures of \( \mathbb{C}\) which
are preserved under translations and the involution \( z\mapsto -z\).

We define the \emph{constant foliation of phase \( \theta\)} to
be the foliation of the complex plane whose
leaves are the straight lines forming an angle \( \pi \theta\) with the
horizontal axis.

It is clear that every constant foliation of \( \mathbb{C}\) is invariant
under translations and the involution \( z\mapsto -z\),
thus it should be possible to define the notion of foliation of phase \( \theta \)
for any quadratic differential.

Following definition \ref{def:metrica-foliacion} we have:

\begin{definicion}
The \emph{foliation of phase \( \theta\)} of a quadratic differential \( \varphi\) is the only
foliation such that in every normal coordinate chart
it is equal to the constant foliation of phase \( \theta\) in \( \mathbb{C}\).

The leaves of this foliation are called \emph{trajectories of phase  \( \theta\)}.
\end{definicion}
Whenever we speak of a trajectory of a differential, we shall mean a trajectory of any phase.
There are several characterization of this trajectories:

\begin{proposicion}
Let \( \varphi\) be a quadratic differential  and  \( \theta\in\mathbb{R}/\mathbb{Z}\).
A curve \( \gamma\) is a trajectory of phase \( \theta\) if and only if
it satisfies any of the following properties:
\begin{itemize}
 \item the function \[ q\mapsto \im\left(e^{-i\pi\theta}\int_p^q\sqrt{\varphi}\right)\]
(which is only locally defined) is constant along \( \gamma\)
\item it is a horizontal trajectory of the quadratic differential  \( e^{-2\pi i \theta}\varphi\)
\end{itemize}

\end{proposicion}
\begin{proof}
	Being a trajectory of phase \( \theta\) and all the other properties are
	local. Hence it is enough to work in a normal coordinate chart.
	Let \( w\) be the function given by \( w(q):= \int_p^q\sqrt{\varphi}\). 
	THe map \( w\) is a normal coordinate chart and in terms of  \( w\)
	the function
	\[ q\mapsto \im\left(e^{-i\pi\theta}\int_p^q\sqrt{\varphi}\right)\]
	equals
	\( \im(e^{-i\pi\theta}w(p))\) which clearly is constant along
	the leaves of the foliation of phase \( \theta\) in the chart \( w\).
	This proves equivalence to the first property.

	To prove the second equivalence observe that the transformation\( \Phi(z):=e^{-i\pi\theta}z\)
	maps lines of phase
	\( \theta\) to horizontal lines, and also satisfies \( \Phi^*(dz^2)=e^{-2i\pi\theta} dz^2\)
	hence the horizontal trajectories of this differential are exactly
	the straight lines of phase
	 \( \theta\).
	Locally, this is the content of the second property.
\end{proof}

\begin{observacion}
	Every trajectory of a quadratic differential is a geodesic of the metric.
	It is also clear that the quadratic differentials
	\( \varphi\) and \( e^{2\pi i \theta}\varphi\) define the same metric.
	Not every geodesic is a trajectory, however, it is possible to prove
	that they are a finite union of trajectories of possibly different
	phases, and the points where the geodesics are not differentiable
	are finite critical points.\footnote{Recall that the metric is singular,
	hence the geodesics are not necessarily differentiable.}
	 \end{observacion}

The \emph{saddle connections} are a special kind of trajectory which will
play a major role later on:

\begin{definicion}
	A \emph{saddle conection} of a quadratic differential  \( \varphi\)
	is a trajectory of some phase \( \theta\) connecting two finite critical points of  \( \varphi\).
	Equivalently, it is a saddle trajectory of the differential  \( e^{-2\pi i \theta}\varphi\).
\end{definicion}

\subsubsection{Horizontal strips}

Recall that theorems \ref{teo:descomposicion} and \ref{teo:triangulacion-wkb}
together imply that for
some kind of quadratic differentials, the regions that appear in the decomposition
are all horizontal strips, \ie, are isomorphic to the interior of sets of the form
\[ F_{2h}:=\{z\in \mathbb{C}| -h\leq\im(z)\leq h\}\]
with the standard differential of the complex plane.

A slightly technical point is that
that the isomorphism does not necessarily extend to the boundary of the strip
(the separating trajectories).
However, we can always extend the isomorphism from \( F_{2h}\)
to the surface and get a continuous (but possible non-injective)
function.

The isomorphism between one of this regions and the set
\( F_{2h}\) is far from being unique.
To see this, it is enough to compute the automorphism group of  \( (F_{2h}, dz^2)\).
Since \( F_{2h}\) is a subset of \( \mathbb{C}\) and we are considering the
differential  \( dz^2\),
the automorphisms are necessarily of the form \( z\mapsto \pm z +b\).
For any of this transformations to preserve the strip
\( F_{2h}\) it is necessary that it fixes, or possibly permutes,
the boundary lines.
This in turn implies that \( \im(b)=0\), \ie
the automorphism group of \( F_{2h}\) is the group of
transformations \( z\mapsto \pm z +r\) with \( r\) real.

In certain cases it is possible to reduce even further the amount of isomorphism
between the horizontal strip and the sets  \( F_{2h}\).
For this, let's assume that every horizontal strip has only one zero
in each component of the boundary and choose any isomorphism \( \psi\) from a set  \( F_{2h}\) to the given strip. Extend the isomorphism to
the boundary of the standard strip.
Thus the inverse image of one of the zeros has imaginary part equal to
some \( h\) and the inverse image of the other has imaginary part equal to \( -h\).
Composing \( \psi\) with a horizontal translation
(which is an automorphism of \( F_{2h}\))
we can ensure that the inverse image of the finite critical points is
symmetric with respect to the origin.

Consider now the set  \( F_{2h}\) 
together with the extra data of a pair of marked points in the boundary
that are symmetric with respect to the origin \ie \( \{w,-w\}\) for some \(w \in \mathbb{C}\).

If we now allow only those automorphisms of \( F_{2h}\)
that map marked points to marked points,
it is clear that the only possibilities is the identity automorphism
and the involution \( z\mapsto -z\).

Motivated by these facts, we give the following definition:

\begin{definicion}
	The \emph{horizontal strip of period \( w\)} is the set
	\[ F_w:=F_{\im(w)}=\{z\in\mathbb{C}| |\im(z)|\leq \frac12\im(w)\}\]
	with the marked points \( \pm \frac12 w\).
	The \emph{standard saddle class} is the line segment through the marked points.
	It is a saddle connection of phase \( \arg(w)\).
	It is clear that \( F_{w_1}\) is isomorphic to \( F_{w_2}\) if and only if  \( w_1=\pm w_2\).

\end{definicion}

The previous remarks immediately imply the next proposition:

\begin{proposicion}
	Let \( \varphi\) be a GMN quadratic differential without
	saddle trajectories and without poles of order greater that two.

	In the decomposition of \( \varphi\) there are only horizontal strips
	with at most two finite critical points on the boundary.

	Each one of the horizontal strips is isomorphic to exactly one
	horizontal strip of some period \( w\)
	and besides, the standard saddle class of \( F_w\) maps to the
	unique saddle connection of  \( \varphi\)
	that is contained in the strip and connects the finite critical points in the boundary.

We shall call this saddle connection the \emph{standard saddle conection} of the horizontal strip.
\end{proposicion}

Finally we want to remark that since the isomorphism from a horizontal strip
to a standard strip in general does not extend to the boundary,
there are different possible topologies for the closure of the strip.
It is not hard to prove that there are exactly two possibilities.
In the figures \ref{fig:franja} and \ref{fig:franja-doblada} we show both
possibilities and the corresponding saddle connections.

\begin{figure}
	\begin{subfigure}[t]{0.45\textwidth}
\hspace{-1cm}
		\includegraphics{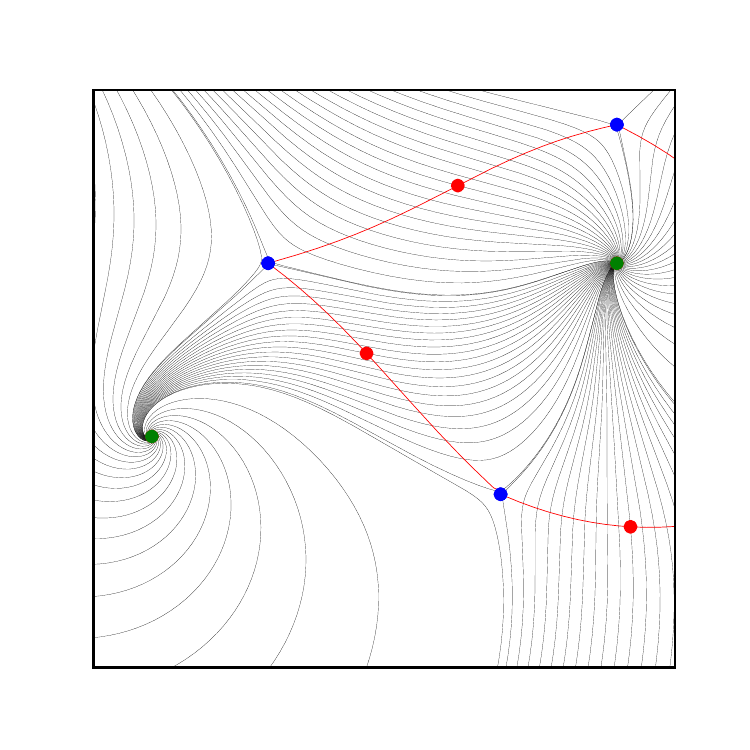}
\caption{Horizontal strip with two finite critical points in the boundary (in blue).
	Some saddle connections are also shown.
}\label{fig:franja}
\end{subfigure}
\hspace{1.0cm}
\begin{subfigure}[t]{0.45\textwidth}
\hspace{-1.5cm}
	\includegraphics{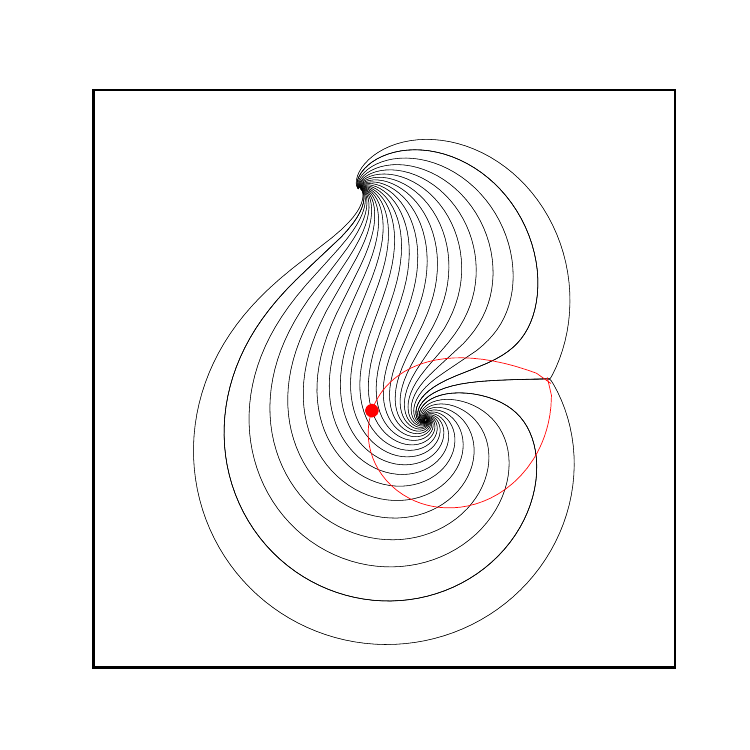}
\caption{Horizontal strip with only one finite critical point in the boundary.
	In this case the saddle connection is a loop.
}\label{fig:franja-doblada}
\end{subfigure}
\end{figure}
\subsection{Periods of GMN differentials}

In this section we will use the standard saddle connections
to give a description of the hat homology group of a quadratic differential.

For the rest of the section, \( X\) will be a fixed Riemann surface and 
\( \varphi\) a fixed GMN differential on \( X\). We also require that
\( \varphi\) does not have saddle trajectories and that all the poles have order
at most two.
As we have already said, this implies that 
\( X\) decomposes as a union of horizontal strips, each of which has at
most two finite critical points in the boundary.

Let \( F\) be any horizontal strip of period  \( w \) of \( X\).
Let \( \alpha\)  be the saddle connection of  \( F\).
Choose any orientation of  \( \alpha\), consider  \( \alpha\) with that orientation as a  \( 1-\)chain on \( X\). Let \( \alpha_1\) and \( \alpha_2\) be the two possible liftings of  \( \alpha\) to the spectral cover of  \( X\). 
Then it is clear that \( \alpha_1-\alpha_2\) and \( \alpha_2-\alpha_1\) are \( 1-\)cycles
that define anti-invariant homology classes, \ie they are elements of 
 \(\widehat\H(\varphi) \),
so they have well defined periods.

To compute this periods, we can integrate \( \psi\) along every summand of the cycle,
and this integral in turn can be done in the standard horizontal strip of \( \mathbb{C}\):
\[ \widehat{\varphi}(\pm(\alpha_1-\alpha_2))=\pm\frac12\left(\int_{\alpha_1}\psi-\int_{ \alpha_2 }\psi\right)=\pm\left( \int_{-\frac12 w}^{\frac12 w}dz  \right)=\pm w\]
where we have used that \( \int_{\alpha_2}-\psi=\int_{\alpha_2}\tau^*(\psi)=\int_{\tau_*(\alpha_2)}\psi=\int_{\alpha_1}\psi\).

Let \( \alpha_F\) be the unique hat homology class given by the cycle \( \pm(\alpha_1-\alpha_2)\) and whose period has positive imaginary part.
We shall call \( \alpha_F\) the \emph{standard saddle class} of the strip \( F\).
Hence, the standard saddle class of a strip of period \( w\) has period equal to \( w\). This is the reason why we choose that name in the first place.

The standard saddle classes are important because they are a basis for
the hat homology group of the differential.
First, we prove that they are linearly independent:

\begin{proposicion}
	Let \( F_1, F_2,\ldots,F_n\) be the horizontal strips of the differential  \( \varphi\) and
	let \( \alpha_i:=\alpha_{F_i}\) be the corresponding standard saddle classes.
	Then the \( \alpha_i\) are linearly independent.
\end{proposicion}
\begin{proof}
Let \( \gamma_i\) be any generic trajectory of the \( i-\)th horizontal strip.
Considering any orientation of \( \gamma_i\) and taking de difference
of the two liftings to the spectral cover \( \widehat X\) we can define a
class \( [\gamma_i]\) in the relative homology group \( \H_1(\widehat X, \CritInf)\).
Using the bilinear pairing discussed in the previous section, it is clear that
\( \langle \gamma_i, \alpha_j\rangle\) is not zero if and only if \( i=j\).
So we can conclude that the classes \( \alpha_i\) are linearly independent.

\end{proof}

To see that the \( \alpha_i\) are a basis, it is enough to prove that there
is the right number of them.
Recall that by proposition
 \ref{prop:dimension-h-gorro} the dimension of  \( \widehat{H}(\varphi)\)
is exactly \( N(\varphi)=6g -6 + \sum_{i=0}^m m_i+m\).

The interested reader can find the proof that there are exactly \( N(\varphi)\)
horizontal strips in \cite[Lemma 3.2]{bridgeland2015}.

Putting this together with the last proposition, we get:
\begin{corolario} Let \( \varphi\) be a GMN quadratic
	differential without poles of order greater than  \( 2\) and without saddle trajectories.
	Then, the standard saddle classes \( \alpha_i\) are a basis for
	the hat homology group \( \widehat\H(\varphi)\).
\end{corolario}

%% file: spaces.tex
\section{Spaces of quadratic differentials}
\label{sec:espacios}

In this section we shall use the standard saddle classes
and the periods of these to define natural parametrizations
of families of quadratic differentials.

\subsection{Horizontal strip decomposition}

The purpose of this section is to describe certain families
of quadratic differentials for which we will later give a parametrization.
In what follows we will only consider GMN differentials with poles of order
at most two, and since we will deal with several distinct differentials
at the same time, we employ the notation \( (X,\varphi)\) for
refering to a quadratic differentials \(\varphi\) on a Riemann surface  \( X\).

\begin{definicion}
We shall say that two quadratic differentials \( (X,\varphi)\) and \( (Y,\vartheta)\)
have the same \emph{(strip) decomposition type} when there is
a diffeomorphism  \( f:X \rightarrow Y\)
that takes poles to poles, zeros to zeros, and separating trajectories into
separating trajectories.
Any such diffeomorphism necessarily establishes a bijection between
the horizontal strips of \(\varphi\) and the strips of \(\vartheta\).
A diffeomorphism with the aformentioned properties will
be called a \emph{strip equivalence}.
\end{definicion}

We are interested in describing 
the set of differentials with the same decomposition type as a given fixed differential.
Besides, one of the main goals is to endow said set with a topology,
and if possible, that this topology is compatible with other
geometric structures, for example the structure of a complex manifold.

However, the problem with this set is that
the more natural attempts to define a topology produce
a singular space \ie instead of defining a manifold, they define
the structure of a complex orbifold.

This problem arises from the fact that given two quadratic differential
with the same decomposition type, in general there is no natural
choice for the strip equivalence.
Equivalently, the problem is that in general there are quadratic differentials
with non-trivial automorphism group.

There are several equivalent ways to solve this problem.
One of those is to consider quadratic differentials with some extra structure:

\begin{definicion}\label{def:marco-homologico}
Let \( (X,\varphi)\) be a quadratic differential.
A \( \Gamma-\)framing \( \theta \) for the differential
is a group isomorphism
\[ \theta:\Gamma\rightarrow \widehat{\H}(\varphi)\]
where clearly \( \Gamma \cong\mathbb{Z}^{N(\varphi)}\).
Let \( (X_1,\varphi_1)\) and \( (X_2,\varphi_2)\) be two differentials with
framings \( \theta_1\) and \( \theta_2\)
and let\( f:X_1\rightarrow X_2\) be a complex isomorphism such that \( f^*(\varphi_2)=\varphi_1\).
We say that \( f\) preserves the framings
when the distinguished lifting
\( f:\widehat X_{\varphi_1} \rightarrow \widehat X_{\varphi_2}\)
that preserves the tautological
\( 1-\)forms makes the following diagram commute:
\begin{center}
\begin{tikzpicture}
    \node (A) at (0,0) {$\widehat \H (\varphi_1)$};
    \node (B) at (3,0) {$\widehat \H (\varphi_2)$};
\node (C) at (1.5,2) {$\Gamma$};
    \draw[transform canvas={yshift=0.3ex},->] (C) -- node[above] {$ \theta_1 $} (A);
    \draw[->] (C) -- node[auto] {$  \theta_2  $} (B);
    \draw[->] (A) -- node[below] {$  \widehat f_*  $} (B);
\end{tikzpicture} 
\end{center}
\end{definicion}

The notion of decomposition type can be extended to the context of
framed differentials, and with this concept
it is possible to prove that the set of
framed differentials with a given decomposition type is a
complex manifold which in turn is a covering
of the orbifold of (non-framed) quadratic differentials.

This approach is fully developed in  \cite{bridgeland2015}.
In this text, we shall use a different method, which is totally equivalent,
but is better suited to our specific needs.
In the next section we shall prove the equivalence of both methods.

In what follows, fix a quadratic differential  \( (X_0,\varphi_0)\).

\begin{definicion}
A \emph{\( \varphi_0-\)framed} differential is a quadratic differential  \( (X,\varphi)\) together with a strip equivalence 
\( f:X_0\rightarrow X\).
\end{definicion}

As in the previous definition, we need a way to determine when
a transformation preserves the \( \varphi_0-\)framings.

\begin{definicion}
	Let \( (X_1, \varphi_1, f_1)\) and \( (X_2, \varphi_2, f_2)\)
	be two \( \varphi_0-\)framed differentials.
	Let {\(g:X_1 \rightarrow X_2\)} be a strip equivalence.
	We shall say that \( g\) is a \emph{framed} strip equivalence
	when the transformation  \( f_2^{-1}\circ g\circ f_1\)
	(which preserves the strip decomposition of \( \varphi_0\))
	fixes the strips, \ie the image of every strip is itself.

	Whenever \( g:(X_1, \varphi_1):\rightarrow (X_2,\varphi_2)\)
	is an isomorphism of quadratic differentials and is also
	a framed strip equivalence, we shall say that it is a
	(framed) isomorphism between \( (X_1, \varphi_1, f_1)\) and \( (X_2, \varphi_2, f_2)\).
\end{definicion}

Let us remark that this way to tackle the problem
of defining the space of quadratic differentials is
entirely analogous to one of the possible constructions of
the Teichm\"uller space of genus \( g\) Riemann surfaces.
See {\cite[Sect.~6.4]{hubbard2006teichmuller}}.

Write \( \Cuad(X_0, \varphi_0)\) for the set of isomorphism classes
of \( \varphi_0-\)framed differentials:

\[ \Cuad(X_0, \varphi_0):= \left\{\vphantom{\frac12} (X,\varphi,f)\, \middle| \, (X,\varphi,f)\mbox{ is a }\varphi_0\mbox{-framed differential}  \right\}/ \cong\]

Observe that \( (X_0,\varphi_0, \Id)\) is a distinguished point in \( \Cuad(X_0, \varphi_0)\).

At the end of this section we shall prove that there is a natural
bijection from
\( \Cuad(X_0, \varphi_0)\) to the set \( \mathbb{H}_+^{N(\varphi)}\) 
given by the \emph{period mapping},
where \( \mathbb H_+ =\{z\in\mathbb{C} | \im(z)>0\}\) is the open upper half-plane.

\subsection{Period coordinates on a cell}

Before being able to define the period mapping,
we need a construction related to the framed strip equivalences.

Let \( (X_1, \varphi_1, f_1)\) and \( (X_2, \varphi_2, f_2)\) be
two \( \varphi_0-\)framed differentials
and let \(g:X_1 \rightarrow X_2 \) be a framed strip equivalence.
Since  \( g\) preserves the ramification points of the corresponding
spectral covers, \( g\) can be lifted to
the coverings.

Since the spectral covers are two-sheeted,
there are exaclty two ways (\apriori indistinguishable) to lift the map \( g\).

Let \( \hat g\) be any of the liftings and let \( F\)
be a horizontal strip of \( X_1\).
Let \( \gamma\) be the saddle connection of  \( F\).
Since the interior of the strip is contractible,
and \( g\) preservs the strips,
we can conclude that \( g(\gamma)\) is homotopic to
the saddle connection of \( g(F)\).
This implies that the lifting \( \hat g\) satisfies \( \hat g_*(\alpha_f)=\pm \alpha_{g(F)}\).

\begin{proposicion}\label{prop:levantamiento-distinguido}
	Let \( F_1,\ldots,F_N\) be the horizontal strips of \( (X_1,\varphi_1)\) 
	and \( \alpha_1,\ldots,\alpha_N\) be the corresponding standard saddle classes.
	One of the liftings of \( g\), which we shall call \( \hat g_+\)
	satisfies:
	\[ \hat g_+ (\alpha_i)= +\alpha_{g(F_i)}\]
	for every standard saddle class, while the other possible lifting
	which we will call
	 \( \hat g_-\) satisfies: 
	\[ \hat g_- (\alpha_i)= -\alpha_{g(F_i)}.\]

	We will also call \( \hat g_+\) the \emph{distinguished lifting} of \( g\).
	For brevity we write \( \hat g:= \hat g_+\).

\end{proposicion}

\subsubsection{The period mapping}

From here on, we shall fix not only the differential 
\( (X_0, \varphi_0)\),  but also an ordering of the horizontal strips of  \( \varphi_0\).

Let \( F_1,\ldots,F_N\) be the \( N:=N(\varphi_0)\) horizontal strips of \( \varphi_0\),
and let
\( \alpha_i:= \alpha_{F_i}\) be the \( N\) standard saddle classes
corresponding to the horizontal strips.

\begin{definicion}
	The \emph{period mapping} defined on the set of
	\( \varphi_0-\)framed quadratic differentials 
	is the map:
	\begin{align*}
		Per: \Cuad(X_0,\varphi_0) & \xrightarrow{~~~~~~} \mathbb{H}_+^N\\
	(X,\varphi, f) &\mapsto  (\hat\varphi (\hat f_*\alpha_i))_{i=1}^N
\end{align*}
where \( \hat f :\hat X_0 \rightarrow \hat X\) is the distinguished lifting of the map \(f \)
discussed earlier and
\( \hat \varphi:\hatH (\varphi)\rightarrow \mathbb{C}\) 
is the period of  \(\varphi\).%
\footnote{See page \pageref{pag:periodo}.}

Let \( \alpha\) be any standard saddle class. Recall that by definition, 
\( \hat \varphi(\alpha)\in \mathbb{H}_+\)
and since \( \hat f_*(\alpha_i)\) is also a standard saddle class,
we conclude that the image o \( Per\) is contained in \( \mathbb{H}_+^N\).
\end{definicion}

The rest of the section is devoted to the proof of the next theorem:

\begin{teorema}\label{teo:principal}
Let \( \varphi_0\) be a quadratic differentials on a Riemann surface \( X_0\)
with simple zeros, double poles, at least one zero and a pole, and no saddle trajectories.
Choose any ordering of the horizontal strips of \( (X_0,\varphi_0)\).
Then, the resulting period mapping
 \( Per:\Cuad(X_0,\varphi_0)\rightarrow \mathbb{H}_+^N\) is a bijection.
\end{teorema}

\subsection{\( \varphi_0-\)framed and \( \Gamma-\)framed differentials}

As we have previously mentioned, the method we are using
to ``frame'' the quadratic differentials is different
from the one used in  \cite{bridgeland2015}.
In this subsection we explore the relationship between the two methods
and prove that they are equivalent.

Recall that we defined the notion of isomorphism
between \( \Gamma-\)framed differentials (see \ref{def:marco-homologico}),
however, we need the notion of strip equivalence for \( \Gamma-\)framed differentials.

A consequence of proposition \ref{prop:dimension-h-gorro} is that
the rank of the hat homology group of the differential
\( (X,\varphi)\), \ie \( N(\varphi)\), only depends
on the genus of the surface \( X\)  
\emph{polar type} of \( \varphi\), \ie the unordered collection
of integers given by the orders of the poles of  \( \varphi\).
This means that if two quadratic differentials are defined on surfaces
of the same genus and they have the same polar type,
then they have isomorphic hat homology groups.
This statement can be rephrased as follows:
the differentials admit \( \Gamma-\)framings, where \( \Gamma = \mathbb{Z}^{N}\).
Thus, we can give the next definition:

\begin{definicion}
Let \( g\) be a non negative integer and \( m=\{m_1,\ldots,m_d\}\) an unorderd collection of possitive integers.
Let \( \Gamma=\mathbb{Z}^N\) where \( N\) is given by the formula of Proposition  \ref{prop:dimension-h-gorro}
taking \( g\) as the genus and \( m\) as the orders of the poles.
The set of GMN \( \Gamma-\)framed differentials is the set:
\[ \Cuad(g,m)^\Gamma:=\raisebox{-.8em}{\( \bigslant{\left\{(X,\varphi, \theta) \middle\vert \begin{array}{l}
			g(X)=g, \varphi \text{ is GMN and of polar type }m\\
			\text{ y }%
		\theta:\Gamma \rightarrow \hatH(\varphi)\text{ is an isomorphism}\end{array}%
\right\}}{\cong} \)}\]

Also, define  \( \Cuad(g,m)\) as:
\[ \Cuad(g,m):=\raisebox{-.8em}{\(  \bigslant{\left\{(X,\varphi) \middle\vert \begin{array}{l}
			g(X)=g, \varphi \text{ is GMN and of polar type }m\\
       \end{array}%
\right\}}{\cong} \)}\]
\end{definicion}

The group \( \Aut(\Gamma)\) acts on \( \Cuad(g,m)^\Gamma\) by pre-composing
the framing with the automorphism.
It is clear that the quotient is isomorphic to \( \Cuad(g,m)\).
Denote by \( p:\Cuad(g,m)^\Gamma\rightarrow\Cuad(g,m)\) the projection
to the quotient,
which actually is the map that forgets the \( \Gamma-\)framing.

In the subset of those differentials in \( \Cuad(g,m)^\Gamma\)
withoyt saddle trajectories we can define
the notion of (strip) decomposition type: 
We shall say that two \( \Gamma-\)framed differentials 
\( (X_1,\varphi_1,\theta_1)\) and \( (X_2,\varphi_2,\theta_2)\)
have the same decomposition type when there is a diffeomorphism
 \( g:X_1\rightarrow X_2\)  that preserves the strips
 and such that the distinguished lifting \( \hat g\) commutes with the framings.
 Observe that it is the same property defining the isomorphisms of \( \Gamma-\)framed diffeomorphism,
but in this case, the diffeomorphism is not necessarily a complex isomorphism
nor preserves the quadratic differential.
The distinguished lifting exists by Proposition
 \ref{prop:levantamiento-distinguido}.

Let \( U\subseteq \Cuad(g,m)\) be the set of quadratic differentials with
the same decomposition type and let
\( U^\Gamma := p^{-1}(U)\subseteq \Cuad(g,m)^\Gamma\).
Let \( (X,\varphi, \theta)\in U^\Gamma\). Write \( \Cuad(X,\varphi,\theta)^\Gamma\)
for the set of \( \Gamma-\)framed differentials that have the same
decomposition type as \( (X,\varphi,\theta)\).
It is clear that \( \Cuad(X,\varphi,\theta)^\Gamma\) maps onto the set \( U\).
Besides, the images of the sets \( \Cuad(X,\varphi,\theta)^\Gamma\) under the
action of \( \Aut(\Gamma)\) is all of \( U\).

In \cite{bridgeland2015} it is proven that \( \Cuad(g,m)^\Gamma\)
is naturally a complex manifold and \( \Cuad(g,m)\) is a complex orbifold.
To relate this formalism to the notion of \( \varphi-\)framed differentials
we will use the next Proposition:
\begin{proposicion}\label{prop:gamma-phi}
	Let \( (X_0,\varphi_0,\theta_0)\in \Cuad(g,m)^\Gamma\).
	Then there is a bijection between
	\( \Cuad(X_0,\varphi_0,\theta_0)^\Gamma\) and
	\( \Cuad(X_0,\varphi_0)\).
\end{proposicion}
\begin{proof}
	Let \( (X,\varphi,f)\in\Cuad(X_0,\varphi_0) \).
	Recall that by Proposition \ref{prop:levantamiento-distinguido}
	\( f\) has a distinguished lifting to the spectral
	covers of  \( X_0\) and \( X\)
	and that this lifting induces an isomorphism of
	the corresponding hat homology groups.
	Let \( \hat f_*\) be that isomorphism.
	Thus \( \hat f_*\circ\theta\) is a
	\( \Gamma-\)framing for \( (X,\varphi)\).
	It can be shown that \( (X, \varphi, \hat f_*\circ\theta) \in \Cuad(X_0,\varphi_0,\theta_0)^\Gamma\)
	and that the resulting \( \Gamma-\)framed differential is well-defined.

	Now, let \( (X, \varphi, \theta) \in \Cuad(X_0,\varphi_0,\theta_0)^\Gamma\).
	There is a diffeomorphism \( f:X_0 \rightarrow X\) that
	preserves the strips and such that the distinguished lifting
	presreves the \( \Gamma-\)framings.
	It is clear that \( (X,\varphi, f)\)
	is a \( \varphi_0-\)framed differential
	We have now to prove that this differential
	is independent of the chosen diffeomorphism \( f\).
	For this, assume that \( g\) is another diffeomorphism
	with the same properties.
	Then \( (X,\varphi, f)\) and \( (X,\varphi, g)\)
	define possibly different \(\varphi_0-\)framed differentias.
	To see that they are the same, we will prove that they have the same
	period vector and appeal to Theorem  \ref{teo:principal}.
	Let \( F_i\) be the \( i\)th horizontal strip of \( X_0\).
	To ensure that the \( i\)th period of the two differentials is the same,
	it is enough to see that \( f(F_i)\) and \( g(F_i)\)
	are the same strip of \( X\).
	Let \( \alpha_i\) be the standard saddle class corresponding to
	\( F_i\) and let \( a:= \theta_0^{-1}(\alpha_i)\)
	be the corresponding element in \( \Gamma\).
	Since the liftings of \( f\) and of \( g\) preserve the \( \Gamma-\)framings,
	it is clear that \( \hat f_*(\alpha_i) = \hat f_*(\theta_0(a))=\theta(a)= \hat g_*(\theta_0(a))= \hat g_*(\alpha_i)\) thus \( \hat f_*(\alpha_i)= \hat g_*(\alpha_i)\) and the corresponding strips are also equal, \ie \( f(F_i)=g(F_i)\).
\end{proof}

\subsection{Gluing of horizontal strips}

This section is devoted to the proof of theorem \ref{teo:principal}.
We shall split the argument in two parts: proving that for every
period there is a differential with that period (surjectivity) and proving
that the period determines the differential (injectivity).

To prove surjectivity, we will use a kind of ``surgery''
to get a differential with arbitrary periods in \( \mathbb{H}_+^N\).

The idea is to take as a starting point the distinguished point of \( \Cuad(X_0,\varphi_0)\) \ie \( (X_0,\varphi_0,\Id)\), and make a 

Let \( V\in \mathbb{H}_+^N\) be an arbitrary vector.
We need to find a \( \varphi_0-\)framed differential whose period vector is
\( V\).

As a first step, we work individually in each strip of \( (X_0,\varphi_0)\).
Without loss of generality, let \( F\) be the first horizontal strip of  \( X_0\).

In this argument it is better to work with standard strips such that
 the marked points in the bottom boundary is zero.
 For every standard strip there is a unique
translation of the complex plane that carries the strip to one with the bottom marked point equal to zero.
After this modification, the strip is no longer symmetric with respect to the origin,
but with respect to the point \( w/2\), where \( w\) is the period of the strip.

Let \( f\) be any isomorphism of the horizontal strip with some standard strip \( F_w\).
And let \( F_{v_1}\) be the standard strip of period  \( v_1\).

We claim there is a diffeomorphism  \( g:F_{v_1} \rightarrow F_w\) with the following properties:
\begin{itemize}
		 \item in neighbourhoods of the boundary of  \( F_{v_1}\) it is a translation of the complex plane;
		 \item it commutes with every translation in the horizontal direction;
		 \item fixes zero and carries  \( v_1\) to \( w\).
\end{itemize}

To construct such diffeomorphism, first we take any smooth increasing function

\( h:[0,\im(v_1)] \rightarrow [0,1]\) that vanishes in a neighbourhood of \(0\) and is
equal to one in a neighbourhood of  \( \im(v_1)\).

With the help of the function \( h\) we defined the sought-after diffeomorphism by the formula:

\[g(z):= z+ (w-v_1)h(\im(z)) \]

It is clear that \( g\) has all the required properties.

Let \( \psi= f\circ g:F_{v_1}\rightarrow X_0\).
Then \( \psi\) is a diffeomorphism between the standard horizontal strip of period
\( v_1\) and the first horizontal strip of  \( X_0\).
This diffeomorphism preserves the horizontal strip, however, in general it is neither a 
complex isomorphism nor preserves the quadratic differential.

By construction the restriction of \( \psi\) to a small enough neighbourhood
of the boundary of \( F_{v_1}\) is in fact a complex isomorphism that
preserves the quadratic differential.

Using \( \psi\), we can copy the quadratic differential and complex structure from the
horizontal strip to the surface  \( X_0\).
In this way we get a new complex structure and a new quadratic differential defined on \( F\)
and such that \( \psi^{-1}\) is an isomorphism to the horizontal strip  \( v_1\).

Iterating this procedure for all the other horizontal strips and periods \( v_i\),
we get new complex structures and quadratic differentials defined for every horizontal strip of \( X_0\).

The problem now lies in proving that this new structures are compatible and define a complex structure 
 on the whole surface \( X_0\), and also that the resulting meromorphic differential has the same poles and zeroes as  \( \varphi_0\).

Observe that every one of this new structures is equal to the old one in neighbourhoods of the saddle trajectories. Since they are equal to the old one, they are compatible amongst each other along small enough neighbourhoods 
of the saddle trajectories and the finite critical points.

As a consequence they are equal on the open set \( X_0^\circ\) and define a complex structure on this
subsurface. So we only need to prove that this structure can be extended to the poles of \( \varphi_0\).
This is equivalent to proving that, in the new complex structure, small punctured neighbourhoods of the poles are
biholomorphic to punctured discs.

\begin{lema}\label{lema:extension-polos}
	Let \( X^\circ_\mathcal{N}\) be the subsurface with the new complex structure and let
	\( p\) be a pole of\( \varphi_0\).
	There is a punctured neighbourhood of
	 \( p\) in \( X^\circ_\mathcal{N}\) that is biholomorphic to a punctured disc
	 \( \mathbb{D}\setminus \{0\}\).
\end{lema}
\begin{proof}
	The proof uses the classification of the possible
	complex structures of  \( \mathbb{D}\setminus\{0\}\) in terms of the modulus.
	For this we refer the reader to \ref{sub:modulus}.
	Specifically, to the definition of the modulus of an annulus (definition \ref{def:modulo}) and to the definition of quasiconformal mappings (definition \ref{def:casiconforme}).

	Let \[ U:=\bigcup_{\overline{F_i}\ni p}^N F_i\] be the union of all the
	horizontal strips that have  \( p\) in their closure.
	Then \( U\) is a neighbourhood of  \( p\).

	Let \( D\) be a small enough disk centered at \( p\) and contained in  \( U\).
	It is clear then that  \( D\setminus \{p\}\) in the old complex structure
	is biholomorphic to a punctured disk. By theorem \ref{lema:modulo-invariante},
	it has infinite modulus.

	For clarity, we shall denote \( U\) endowed with the new complex structure by  \( U_\mathcal{N}\)
	and with the old complex by \( U_\mathcal{O}\).
	
	The identity mapping  \( \Id: U_\mathcal{O}\rightarrow U_\mathcal{N}\)
	is a quasiconformal mapping.
	To see this, it is enough to prove that it is quasiconformal in every
	horizontal strip, because proposition \ref{prop:pegado-casiconforme} then
	ensures that it is quasiconformal in the union of the strips.

	To see that the identity map is quasiconformal in every strip,
	observe that it is equivalent to the diffeomorphism \( \psi\) with
	wich we defined the new complex structure.

	Since \( \psi\) is differentiable and commutes with horizontal translations it is enough to take the supremum of the pointwise dilation in any vertical line.
	The strip is bounded on the vertical direction so we can conclude that
	\( \psi\) has bounded dilation and hence is quasiconformal.
	This proves that \( \Id: U_\mathcal{O}\rightarrow U_\mathcal{N}\) is quasiconformal, hence its restriction to the disc \( D\setminus\{p\}\)
	is also quasiconformal.

	Given that the modulus is a quasiconformal quasi-invariant, the modulus of \( D\setminus \{p\}\) with the old and new complex structure satisfy:
	\[ \frac 1 K m\big(D_\mathcal{N}\!\setminus\!\{p\}\big) ~~\leq~~ m\big(D_\mathcal{O}\!\setminus\! \{p\}\big) ~~\leq~~ K m\big(D_\mathcal{N}\!\setminus\! \{p\}\big)\]
	where \( K\) is a possitive constant. Since \( m\big(D_\mathcal{O}\setminus \{p\}\big)= \infty\), we conclude that \( m\big(D_\mathcal{N}\setminus \{p\}\big)=\infty\)
	\ie , \( D_\mathcal{N}\setminus \{p\}\) is biholomorphic to  \( \mathbb{D}\setminus \{0\}\),
	as required.
\end{proof}

This proves that the complex structure of \( X_0^\circ\) can be extended to the whole surface. In a similar fashion we can prove that the quadratic differentials are compatible and define a holomorphic quadratic differential on  \( X_0^\circ\).
This holomorphic differential extends to a meromorphic quadratic differential
on \( X_0\).
Hence we only need to prove that the orders of the
poles of this new differential are the same as those of \( \varphi_0\).
Since the horizontal trajectories of \( \varphi_0\), are also trajectories of
the new differential, and by the classification of the local behaviour
of the trajectories around a pole, the orders of the poles must be the same.
Thus we get a new \( \varphi_0-\)framed differential \( (X_\mathcal{N},\varphi_N,\Id)\)
whose period vector is exactly \( V\).

This then completes the proof of the surjectivity of the period mapping.

To complete the prove of the theorem, we need to prove injectivity of
the period mapping. Let  \( (X_1,\varphi_1, f_1)\) and \( (X_2,\varphi_2, f_2)\)
be two \( \varphi_0\)-framed differentials with the same periods.

We need to define a \( \varphi_0-\)-framed isomorphism betwee \( (X_1,\varphi_1)\) and \( (X_2,\varphi_2)\)
\ie an isomorphism \( g:X_1\rightarrow X_2\) such that for every strip
 \( F\) of \(X_0\),
\( g\circ f_1(F)\) and \( f_2(F)\) are the same strip.

As in the proof of surjectivity, we shall work individualy in every
horizontal strip of \( X_1\).

Let \( F\) be any horizontal strip of  \( X_0\) and let \( F_1:=f_1(F)\) and \( F_2:=f_2(F)\) be the corresponding strips in  \( X_1\) and \( X_2\).
Choose an oriented curve \( \gamma\) connecting the two finite critical points
in the boundary of  \( F\),
and consider the its image in \( X_1\) and \( X_2\).

There is a unique isomorphism \(\psi_1 \) between \( F_1\) and \( F_w\) where\( w\)
is the period of  \( F_1\) and such that the finite critical points are
mapped to the marked points and also the chosen curve is mapped to a
curve homotopic to the oriented segment from zero to \( w\).
Let \( \psi_2\) be the unique isomorphism between \( F_2\) and \( F_w\)
with the same properties as \(\psi_1 \).

Compossing the isomorphisms \( \psi_1\) and \( \psi_2^{-1}\), we get an isomorphism
from  \( F_1\) to
\( F_2\) sending the curve \( f_1(\gamma)\) in a curve homotopic to  \( f_2(\gamma)\).
For every horizontal strip we define an isomorphism as above.
By construction all of these are
compatible on the critical points of the differential.
Let us prove that they are compatible in every separating trajectory.

Let \( \sigma\) be a separating trajectory that is in the boundary of two
horizontal strips and such that it tends to a point  \( p\), which is a zero of
the differential.
Every point \( q\in\sigma\) is at a finite distance from \( p\) in the
metric of  \( \varphi_1\), and this distance is invariant under isomorphisms of
the quadratic differential. So we can conclude that the image of \( q\) under
every one of the isomorphisms defined on the horizontal strips that contain
 \( \sigma\) in their boundary, coincides, and this is true for every separating
 trajectory.

Since the isomorphisms are compatible on the boundaries of every horizontal strip,
they are compatible globaly and define a homemorphism
\( g\) that is a complex isomorphism on the interior of every strip.
This implies that the homeomorphism is actually a biholomorphism.
Since \( g^*(\varphi_2)= \varphi_1\) on the interior of every horizontal strip,
and given that  \( g^*(\varphi_2)\) and \(\varphi_1\) are meromorphic sections of
the same line bundle, this is enough to prove that they are actualy the same.
Finally we conclude that \( g\) is an isomorphism of \( \varphi_0-\)framed differentials, which concludes the proof of theorem \ref{teo:principal}.

Theorem \ref{teo:principal} is equivalent to Proposition  4.9 of \cite{bridgeland2015}
and the proof given here is adapted from the proof given by Bridgeland and Smith.

\subsection{Digression: quasiconformal mappings and the modulus of an annulus}
\label{sub:modulus}

In this section we define the notion of \emph{modulus} of an annulus and 
the class of \emph{quasiconformal mappings}.
We also recall the theorems we used in the proof of lemma
\ref{lema:extension-polos}.
Given that this is a vast subject outside the scope of
this text, we shall refrain from giving the proofs of this theorems.
One of the classical references for this topics is \cite{ahlfors},
where all the proofs can be found.
Another more up-to-date reference is \cite{hubbard2006teichmuller}.
The content of this section is well-known and is one of the possible
starting points for the study of  \emph{Teichm\"uller spaces}.

Let \( A\) be a Riemann surface that is diffeomorphic to the punctured unitary disc.
This kind of surfaces are called \emph{annuli}.
The universal covering of an annulus is a Riemann surface diffeomorphic to
the complex plane, and the group of covering transformations (which is isomorphic to \( \mathbb{Z}\)) acts holomorphically.

In this case, as a consequence of the uniformization theorem and
the classification of M\"obius transformations,
the universal covering  is then the hyperbolic plane or the complex plane.%
\footnote{There is no free holomorphic action of \(\mathbb{Z}\) on the Riemann sphere.}

In the first case, it is easy to see that the action is equivalent to

\begin{align*}
	\mathbb{Z}\times \mathbb{C} & \xrightarrow{~~~~~} \mathbb{C}\\
n,z & \mapsto z+2\pi in
\end{align*}

so \( A\), which is the quotient, is isomorphic to the complex torus \( \mathbb{C}^*\).

In the second case, there are two options:

Either the generator of the group acts as a parabolic transformation, or
as a hyperbolic transformation.
If the generator acts as a parabolic transformation, then the action is equivalent to:

\begin{align*}
	\mathbb{Z}\times \mathbb{H} & \xrightarrow{~~~~~} \mathbb{H}\\
n,z &\mapsto z+n
\end{align*}
and in this case, the quotient is isomorphic
 to the annulus
\( \mathbb{A}_\infty:= \{z\in \mathbb{C} | 1 \leq | z |\}\)

If the generator acts as a hyperbolic transformation,
there is a positive real number \( r\) such that the action is equivalent
 to:

\begin{align*}
	\mathbb{Z}\times \mathbb{H} & \xrightarrow{~~~~~} \mathbb{H}\\
n,z &\mapsto e^{(nr)}z
\end{align*}
or equivalently
\begin{align*}
	\mathbb{Z}\times \mathbb{B} & \xrightarrow{~~~~~} \mathbb{B}\\
n,z &\mapsto nrz
\end{align*}
where \( \mathbb{B}:=\{z\in \mathbb{C}| 0 \leq \im(z) \leq \pi\}\) 
is the band model of the hyperbolic plane.

Using this last model, it can be shown that the transformation
\(z\mapsto e^{-2\pi i (z/r)} \) gives an isomorphism between the quotient and the
annulus
\( \mathbb{A}_{M}:= \{z\in \mathbb{C} | 1 \leq | z | \leq e^{2\pi M}\}\)
where \( M:=\pi/r\).

The previous observations constitute a sketch of the classification of
all Riemann surfaces topologically equivalent to the punctured disc up to complex
isomorphism.

\begin{definicion}\label{def:modulo}
The number \( M\) in the definition of the annulus \( \mathbb{A}_M\)
is called the \emph{modulus} and is invariant under complex isomorphism.
When \( M=\infty\) we will say that the annulus is \emph{infinite}.
If the annulus is biholomorphic to \(\mathbb{C}^* \) we shall say that it is
\emph{doubly infinite}.

We usually write \( m(A)\) for the modulus of the annulus \(A\).
\end{definicion}

It is clear that \( \mathbb{D}\setminus \{0\}\) is an infinite annulus.

The next lemma summarizes the previous remarks:

\begin{lema}\label{lema:modulo-invariante}
	The modulus is a full complex (or conformal) invariant of annuli.
\end{lema}

For more details on this definition and the proof of the lemma,
see \cite[p.~11]{ahlfors} or \cite[prop. 3.2.1]{hubbard2006teichmuller}.

As we have already indicated, the modulus is a conformal invariant of
the annuli, however, in certain circumstances, the class of conformal
maps is too small to be useful.

One way to solve this issue is to instead consider \emph{quasiconformal maps}.
There are several equivalent definitions of quasiconformal maps,
none of which is particularly short.

Let \( f\) be a \( C^1\) mapping between open sets of the complex plane
and let \( p \in U\)
be any point in the domain.
The (pointwise) dilation of \( f\) at \( p\) is the number 
\[D_f(p):=\frac{|f_z|+|f_{\overline z}|}{|f_z|-|f_{\overline z}|}(p) \]

Observe that the dilation is at least one, and that is continuous
if the derivative of \( f\) is continuous.
In particular, the dilation of a diffeomorphism
is bounded in any compact subset of its domain.
Remark that the dilation is exactly one if and only if \( f_{\overline z}=0\),
\ie when  \( f\) is conformal at \( p\).

\begin{definicion}\label{def:casiconforme}
	Define the (global) \emph{dilation} of \( f\)
	as the supremum of the pointwise dilations
	over the domain of  \( f\).
	If the dilation of \( f\) is finite, we say that
	\( f\) is a quasiconformal mapping of quasiconformality constant
	 \( K:=sup(D_f(p))\), or more briefly,  that \( K-\)quasiconformal.
\end{definicion}

Every  \( 1-\)quasiconformal map is actually conformal.

The notion of quasiaconformality can be extended to the class of
homeomorphisms, not necessarily differentiable.
For this we need the notion of distributional (or weak) derivatives.

We shall refrain form giving the specific details of this definition, instead we refer the reader
to Chapter 2 of  \cite{ahlfors}
and also Chapter 4 of \cite{hubbard2006teichmuller}.

As we have said, one of the virtues of using quasiaconformal mappings
is that they are more abundant than conformal mappings.
Another useful property is that some of the conformal
invariants give rise to quasiaconformal \emph{quasi-invariants}.
One of the most important examples of these \emph{quasi-invariants} is
precisely the modulus of a cylinder.

\begin{proposicion}[Gr\"otzsch theorem for annuli]
	Let \( f: A\rightarrow A^\prime\) be a  \( K-\)quasiconformal map
	between annuli of modulus
	 \( m\) and \( m^\prime\) respectively.
	Then \( m\) and \( m^\prime\) satisfy:
	\[\frac 1K m^\prime \leq m \leq K m^\prime \]
	It is in this sense that we say that the modulus is a quasiconformal \emph{quasi-invariant}.
\end{proposicion}

For the proof of this proposition see Theorem 1.2 of \cite{ahlfors} or section  4.3 of \cite{hubbard2006teichmuller}.

Observe that by the way in which we defined
the dilation of a homeomorphism,
\ie that it only depends on quantities invariant
under conformal (or holomorphic) coordinate changes,
the dilation can also be defined
for homeomorphisms between Riemann surfaces.
Hence the notion of quasiaconformality can also be extended to
the case of Riemann surfaces.
Finally, the next Proposition allows us to prove
that a mapping is quasiaconformal given that it is
quasiaconformal in a collection of subsets of its domain:

\begin{proposicion}
	Let \( X\) and \( Y\) be two Riemann surfaces and let
	\( f:X\rightarrow Y\) be a homeomorphism.
	Suppose there is a compact subset \( A\) of \( X\) 
	that topologically is an embedded graph.
	If \( f\) is \( K-\)quasiconformal
	on every connected component of \( X\setminus A\), then \( f\) is quasiconformal.
\end{proposicion}

This proposition is a consequence of:

\begin{proposicion}\label{prop:pegado-casiconforme}
Let \( f:U \rightarrow V\) be a homeomorphism such that \( U\) and \( V\)
are subsets of the complex plane.
If \( f\) is quasiconformal in the complement of the real axis,
then it is quasiconformal.
\end{proposicion}
In turn, this proposition is a consequence of proposition
 4.2.7 of \cite{hubbard2006teichmuller}.